\RequirePackage{fix-cm}
\documentclass{svjour3}                     
\smartqed  
\usepackage{graphicx}
\usepackage{amsmath}
\usepackage{amsfonts}
\usepackage{amssymb}
\usepackage{subeqnarray}
\usepackage{subfigure}
\usepackage{epsfig}
\usepackage{multirow}%
\usepackage{enumerate}

\def\ol{\overline}
\def\ie{\textit{i.e.}}
\def\eg{\textit{e.g.} }
\newcommand{\vecloc}[1]{\stackrel{\rightharpoonup}{#1}}
\newcommand{\cevloc}[1]{\stackrel{\leftharpoonup}{#1}}
\begin{document}

\title{A robust high-order Lagrange-projection like scheme with large time steps for the isentropic Euler equations
}

\titlerunning{A high-order Lagrange-projection like scheme}        

\author{Florent Renac}


\institute{F. Renac \at
           ONERA The French Aerospace Lab, 92320 Ch\^atillon Cedex, France \\
           \email{florent.renac@onera.fr}           
}

\date{Received: date / Accepted: date}

\maketitle

\begin{abstract}
We present an extension to high-order of a first-order Lagrange-projection like method for the approximation of the Euler equations introduced in Coquel {\it et al.} (Math. Comput., 79 (2010), pp.~1493--1533). The method is based on a decomposition between acoustic and transport operators associated to an implicit-explicit time integration, thus relaxing the constraint of acoustic waves on the time step. We propose here to use a discontinuous Galerkin method for the space approximation. Considering the isentropic Euler equations, we derive conditions to keep positivity of the mean value of density and satisfy an entropy inequality for the numerical solution in each element of the mesh at any approximation order in space. These results allow to design  limiting procedures to restore these properties at nodal values within elements. Numerical experiments support the conclusions of the analysis and highlight stability and robustness of the present method, though it allows the use of large time steps.
\keywords{Lagrange-projection \and discontinuous Galerkin method \and explicit-implicit \and entropy-satisfying \and positivity-preserving \and isentropic Euler equations}
\subclass{65M12 \and 65M70}
\end{abstract}

%
%
\section{Introduction}

In this work, we are interested in the design of a robust and accurate numerical method for the description of flows exhibiting multiple space and time scales that can differ by several orders of magnitude. An example is the cooling system of high-pressure gas turbines in turbomachinery. The flow induced by the blade rows is transonic, while the internal cooling channels contain region of very low speed convection. Another examples are multiphase flows where the acoustic time scales strongly differ between liquid and gas, or the Euler equations in the incompressible limit where the speed of acoustic and transport waves differ notably.

In those applications, one may be interested in the transport phenomena associated to slow waves only. Unfortunately, numerical methods are usually designed for the resolution of all waves and suffer from restriction of the fast waves. In standard explicit shock-capturing methods the time step is limited by the fast waves to ensure stability of the numerical scheme. Moreover, their numerical diffusion is proportional to the speed of the fastest waves. Both aspects impose over-resolution in time and space. As a consequence, spurious pressure oscillations that deteriorate the quality of the solution are usually observed.

The present work is mainly based on a Lagrange-projection like scheme introduced in \cite{coquel_etal_10} in the context of a first-order finite volume formulation of the Euler equations. The method uses the Lagrange-projection framework \cite{godlewski-raviart96} for the splitting of acoustic and transport operators, but no mesh movement is applied. The acoustic operator is solved in Lagrange coordinates, while the projection onto the grid is replaced by the transport operator. The Lagrange step is integrated in time with an implicit backward-Euler scheme in order to relax the time step restriction associated to acoustic waves. An explicit forward Euler method is applied for the transport step in order to accurately describe associated unsteady phenomena. This method was then successfully used to design an asymptotic preserving scheme for the discretization of the Euler equations with source terms \cite{chalons_etal_13}.

Besides, the work in \cite{coquel_etal_10} circumvent the difficulties in the treatment of nonlinearities associated to the equation of state by using a relaxation approximation \cite{Jin_Xin_95,coquel_etal_01}. The latter method approximates the nonlinear system with a linear or a quasi-linear enlarged system with stiff relaxation source terms. In the limit of instantaneous relaxation, the system is consistent with the original system. Relaxation approximations have been applied to the isentropic Euler equations \cite{chalons_coulombel08} and to the full system of gas dynamics \cite{chalons_coquel_05,coquel_etal_01} where only the physical pressure is replaced by a relaxation pressure with its own evolution equation.

One attractive feature of the Lagrange-projection like scheme introduced in \cite{coquel_etal_10} is the preservation of convex invariant domains and entropy inequality at the discrete level, while allowing the use of large time steps. The objective of the present work is to extend this method to space discretization of arbitrary order by using a discontinuous Galerkin (DG) method \cite{reed-hill73,lesaint-raviart74} which has become very popular for the solution of nonlinear convection dominated flow problems \cite{cockburn-shu89,cockburn-shu01,adigma-book10,renac_etal_2012c,renac_etal15}. These are particular aspects that make the DG method well suited. First, it is possible to use the numerical fluxes derived in \cite{coquel_etal_10} and therefore the present method may be viewed as a natural extension to high-order of the related numerical method. Then, the effect of the numerical flux on the quality of the approximation is known to decrease as the polynomial degree $p$ in the DG method increases \cite{cockburn-shu01,qui_et-al06,renac15a}. This avoids the use of local numerical parameters tuned at each interface of the mesh in order to lower the numerical diffusion induced by the first-order approximation \cite{coquel_etal_10,chalons_etal_13}. This aspect is essential in our analysis to restore the PDE properties at the discrete level and to derive \textit{a priori} conditions to preserve invariant domains and satisfy an entropy inequality by the present Lagrange-projection DG (LPDG) scheme. These two latter properties are satisfied for the mean value of the numerical solution in the elements of the mesh and present similarities with the positivity preserving scheme in \cite{perthame_shu_96} and the entropy satisfying scheme in \cite{berthon_06}. Besides, they suggest the application of \textit{a posteriori} limiters introduced in \cite{zhang_shu_10a,zhang_shu_10b} that extend the properties to nodal values whithin elements.

The paper is organized as follows. Section~\ref{sec:1Dmodel_pb} presents the model problem with the system of isentropic Euler equations (section~\ref{sec:Euler_eqn}), its relaxation approximation (section~\ref{sec:relax_Euler}) and the splitting between acoustic and transport operators (section~\ref{sec:LP_splitting}). The numerical approach for the high-order space discretization is introduced in section~\ref{sec:DG_discr}, while time discretization is described in section~\ref{sec:time_discr}. The first-order implicit-explicit time integration is described in section~\ref{sec:1st_order_time_discr} and the properties of the numerical scheme are analyzed in section~\ref{sec:ana_1st_order_time_discr}. High-order time discretization and limiting strategies are discussed in sections~\ref{sec:HO_time_discr} and \ref{sec:limiters}, respectively. These results are assessed by several numerical experiments in section \ref{sec:num_xp}. Finally, concluding remarks about this work are given in section \ref{sec:conclusions}.

\section{One-dimensional model problem}\label{sec:1Dmodel_pb}

\subsection{Isentropic Euler equations}\label{sec:Euler_eqn}

The discussion in this paper focuses on the Euler equations for an isentropic gas in one space dimension. Let $\Omega=\mathbb{R}$ be the space domain and consider the following problem
\begin{subeqnarray}\label{eq:Euler_eq}
 \partial_t {\bf u} + \partial_x{\bf f}({\bf u})  &=& 0,\hspace{0.925cm} \mbox{in }\Omega\times(0,\infty),\\
 {\bf u}(\cdot,0) &=& {\bf u}_{0}(\cdot),\quad\mbox{in }\Omega.
\end{subeqnarray}

The vector
\begin{equation}\label{eq:cons_var}
 {\bf u}=\left(
 \begin{array}{c}
  \rho\\ \rho u
 \end{array}\right)
\end{equation}

\noindent represents the conservative variables with $\rho$ the density and $u$ the velocity. The nonlinear convective fluxes in (\ref{eq:Euler_eq}a) are defined by
\begin{equation}
 {\bf f}({\bf u}) = \left(
\begin{array}{c}
  \rho u\\
  \rho u^2+\mathrm{p}
\end{array}
\right).
\end{equation}

Equations (\ref{eq:Euler_eq}) are supplemented with an equation of state for the pressure of the form $\mathrm{p}=\mathrm{p}(\tau)$, with $\tau=1/\rho$ the specific volume. Assuming that $\mathrm{p}'(\tau)<0$ and $\mathrm{p}''(\tau)>0$ for all $\tau>0$, the system (\ref{eq:Euler_eq}a) is strictly hyperbolic over the set of states
\begin{equation*}
 \Omega^a=\{{\bf u}\in\mathbb{R}^2:\;\rho>0,u\in\mathbb{R}\},
\end{equation*}

\noindent with eigenvalues $u\pm c(\tau)$ associated to nonlinear fields. The sound speed is defined by $c^2(\tau)=\tau^2e''(\tau)$ with $e'(\tau)=-\mathrm{p}(\tau)$ the specific internal energy.

Introducing the specific total energy $E=e+u^2/2$, the mapping $\rho E:\Omega^a\ni{\bf u}\mapsto\rho E({\bf u})\in\mathbb{R}$ is a strictly convex function \cite{godlewski-raviart96}. Physically relevant solutions to (\ref{eq:Euler_eq}) must hence satisfy an inequality of the form
\begin{equation}
 \partial_t\rho E + \partial_x(\rho Eu+\mathrm{p}u) \leq 0.
\end{equation}

\subsection{Relaxation approximation of the Euler equations}\label{sec:relax_Euler}

Solutions to the isentropic Euler equations (\ref{eq:Euler_eq}a) may be approximated by solutions of the following Suliciu relaxation system \cite{chalons_coulombel08}
\begin{subeqnarray}\label{eq:relax_sys_Euler}
 \partial_t\rho + \partial_x(\rho u) &=& 0,\\
 \partial_t(\rho u) + \partial_x(\rho u^2+\Pi) &=& 0,\\
 \partial_t(\rho \Pi) + \partial_x(\rho \Pi u+a^2u) &=& -\tfrac{\rho(\Pi-\mathrm{p}(\tau))}{\varepsilon},
\end{subeqnarray}
\noindent where $\varepsilon>0$ represents a characteristic relaxation time. Equation (\ref{eq:relax_sys_Euler}c) modelizes the evolution of the pressure relaxation, $\Pi$, in flows subject to mechanical disequilibrium. The variable $\Pi$ may be viewed as a linearization of the pressure $\mathrm{p}$ around its equilibrium, $\Pi=\mathrm{p}(\tau)$, while $a>0$ is a parameter that approximates the Lagrangian sound speed, $\rho c$, and whose value will be specified later. In the limit of instantaneous relaxation, we get
\begin{equation}\label{eq:instant_relax_edp}
 \lim_{\varepsilon\rightarrow0} \Pi = \mathrm{p}(\tau),
\end{equation}

\noindent and system (\ref{eq:relax_sys_Euler}) converges formally toward (\ref{eq:Euler_eq}a). We refer the reader to \cite{chalons_coulombel08} for an in-depth discussion of the relaxation approximation of the Euler equations for an isentropic fluid.

It may be checked that the homogeneous quasilinear system (\ref{eq:relax_sys_Euler}) is strictly hyperbolic over the set of states
\begin{equation}\label{eq:space_states_relax}
 \Omega^r=\{{\bf w}\in\mathbb{R}^3:\;\rho>0,u\in\mathbb{R},\Pi\in\mathbb{R}\},
\end{equation}

\noindent with eigenvalues $\mu_1=u-a\tau$, $\mu_2=u$, and $\mu_3=u+a\tau$. All the characteristic fields associated with these eigenvalues are linearly degenerate.

We note that the relaxation system (\ref{eq:relax_sys_Euler}) is a dissipative approximation of the Euler equations providing that the following subcharacteristic condition is satisfied
\begin{equation}\label{eq:sub_cond}
 a>\max_\tau\sqrt{-\mathrm{p}'(\tau)},
\end{equation}

\noindent for all $\tau$ under consideration \cite{chalons_coulombel08}.

\subsection{Acoustic-transport operator splitting}\label{sec:LP_splitting}

We decompose equation (\ref{eq:Euler_eq}a) between acoustic and transport operators with a sequential splitting: the acoustic step reads
\begin{equation}\label{eq:acq_step_cons}
 \partial_t {\bf u} + (\partial_xu){\bf u} + \partial_x \left(
\begin{array}{c}
  0\\
  \mathrm{p}
\end{array}
\right) = 0,
\end{equation}

\noindent then the transport step reads
\begin{equation}\label{eq:trp_step}
 \partial_t {\bf u} + \partial_x{\bf f}_t({\bf u}) -(\partial_xu){\bf u} = 0,
\end{equation}

\noindent with ${\bf f}_t({\bf u})=u{\bf u}$. Assuming smooth solutions, the acoustic step (\ref{eq:acq_step_cons}) may be rewritten into the equivalent form
\begin{subeqnarray}\label{eq:acq_step}
 \partial_t \tau - \tau\partial_xu = 0,\\
 \partial_t u + \tau\partial_x\mathrm{p} = 0.
\end{subeqnarray}

In the following, we shall consider a Suliciu relaxation approximation for the Lagrangian gas dynamics applied to (\ref{eq:acq_step}). Considering the relaxation system for the Euler equations (\ref{eq:relax_sys_Euler}), the relaxation system for the acoustic part reads 
\begin{equation}\label{eq:acq_step_relax}
 \partial_t {\bf w} + \tau\partial_x{\bf f}_a({\bf w}) = {\bf s}({\bf w}),
\end{equation}

\noindent with
\begin{equation}\label{eq:Lag_var}
 {\bf w} = \left(
\begin{array}{c}
  \tau \\
  u\\
  \Pi
\end{array}
\right), \quad {\bf f}_a({\bf w}) = \left(
\begin{array}{c}
  -u\\
  \Pi\\
  a^2u
\end{array}
\right), \quad
{\bf s}({\bf w}) = \left(
\begin{array}{c}
  0\\
  0\\
  -\tfrac{\Pi-\mathrm{p}(\tau)}{\varepsilon}
\end{array}
\right).
\end{equation}

With a slight abuse, the whole vector ${\bf w}$ will be referred to as the Lagrange variables in the following. It may be shown that the relaxation system (\ref{eq:acq_step_relax}) is a dissipative approximation of the acoustic step (\ref{eq:acq_step}) under the subcharacteristic condition (\ref{eq:sub_cond}).

Let $m$ be the mass variable defined by $dm=\rho_0(x)dx$. Then, approximating the space derivative operator $\tau\partial_x$ by $\tau(x,0)\partial_x$ with $\tau(\cdot,0)=1/\rho_0(\cdot)$ and imposing $\varepsilon\rightarrow\infty$, the equations in (\ref{eq:acq_step_relax}) may be written in the following conservation form
\begin{subeqnarray}\label{eq:acq_step_relax_homogen}
 \partial_t \tau - \partial_mu &=& 0,\\
 \partial_t u + \partial_m\Pi &=& 0,\\
 \partial_t\Pi + \partial_m(a^2u) &=& 0 
\end{subeqnarray}

The following results hold for the homogeneous relaxation system (\ref{eq:acq_step_relax_homogen}).

\begin{theorem}[Hyperbolicity of the relaxation system]
\label{th:RP_hyp_relax}
The homogeneous relaxation system (\ref{eq:acq_step_relax_homogen}) is strictly hyperbolic over the set of states (\ref{eq:space_states_relax}) with eigenvalues $\mu_1=-a<\mu_2=0<\mu_3=a$. All the characteristic fields associated with these eigenvalues are linearly degenerate. Moreover, the characteristic variables associated with these eigenvalues are $\vecloc{w}=\Pi+au$, $J=\Pi+a^2\tau$, and $\cevloc{w}=\Pi-au$, respectively:
\begin{subeqnarray}\label{eq:Riemann_invar_acou}
 \partial_t\vecloc{w}+a\partial_m\vecloc{w} &=& 0,\\
 \partial_t J &=& 0,\\
 \partial_t\cevloc{w}-a\partial_m\cevloc{w} &=& 0.
\end{subeqnarray}

\end{theorem}

\begin{proof}
Computing the eigenvalues and eigenvectors of $\partial{\bf f}_a/\partial{\bf w}$ is straightforward and is left to the reader. Likewise, the three operations (\ref{eq:acq_step_relax_homogen}c)$+a$(\ref{eq:acq_step_relax_homogen}b), (\ref{eq:acq_step_relax_homogen}c)$+a^2$(\ref{eq:acq_step_relax_homogen}a), and (\ref{eq:acq_step_relax_homogen}c)$-a$(\ref{eq:acq_step_relax_homogen}b) give (\ref{eq:Riemann_invar_acou}a,b,c), respectively. \qed
\end{proof}

\begin{theorem}[Riemann problem for the relaxation system]
\label{th:RP_acq_relax}
Consider the Riemann problem defined by the homogeneous relaxation system (\ref{eq:acq_step_relax_homogen}) associated with the initial condition
\begin{equation*}
 {\bf w}(m,0)= \left\{
\begin{array}{ll}
  {\bf w}_L, & m<0,\\
  {\bf w}_R, & m>0,
\end{array}
\right.
\end{equation*}

\noindent where ${\bf w}_L$ and ${\bf w}_R$ are in $\Omega^r$. Then, the unique solution to this problem is the self-similar solution ${\boldsymbol{\cal W}}(\cdot;{\bf w}_L,{\bf w}_R)$ defined by
\begin{equation}\label{eq:sol_PR_relax}
 {\boldsymbol{\cal W}}(\tfrac{m}{t};{\bf w}_L,{\bf w}_R)= \left\{
\begin{array}{lr}
  {\bf w}_L, & \frac{m}{t}<-a,\\
  {\bf w}_L^\star, & -a<\frac{m}{t}<0,\\
  {\bf w}_R^\star, & 0<\frac{m}{t}<a,\\
  {\bf w}_R, & \frac{m}{t}>a,
\end{array}
\right.
\end{equation}

\noindent where ${\bf w}_L^\star=(\tau_L^\star,u^\star,\Pi^\star)^\top$, ${\bf w}_R^\star=(\tau_R^\star,u^\star,\Pi^\star)^\top$, and
\begin{subeqnarray}\label{eq:sol_PR_relax2}
 u^\star   &=& \frac{u_L+u_R}{2}+\frac{\Pi_L-\Pi_R}{2a},\\
 \Pi^\star &=& \frac{\Pi_L+\Pi_R}{2}+\frac{a(u_L-u_R)}{2},\\
 \tau_L^\star &=& \tau_L + \frac{u^\star-u_L}{a},\\
 \tau_R^\star &=& \tau_R + \frac{u_R-u^\star}{a}.
\end{subeqnarray}

\end{theorem}

\begin{proof}
The Rankine-Hugoniot relations associated with the continuity and momentum equations (\ref{eq:acq_step_relax_homogen}a,b) through the steady $2$-wave impose $u_L^\star=u_R^\star=u^\star$ and $\Pi_L^\star=\Pi_R^\star=\Pi^\star$. Now, applying the Rankine-Hugoniot relations across the $1$- and $3$-waves, one obtains
\begin{subeqnarray}\label{eq:RH_relax}
 -a(\tau_L^\star-\tau_L) + (u^\star-u_L) &=& 0,\\
 a(\tau_R-\tau_R^\star) + (u_R-u^\star) &=& 0,\\
 -a(u^\star-u_L) - (\Pi^\star-\Pi_L) &=& 0,\\
 a(u_R-u^\star) - (\Pi_R - \Pi^\star) &=& 0.
\end{subeqnarray}

Adding and subtracting (\ref{eq:RH_relax}c) and (\ref{eq:RH_relax}d), one obtains the expressions for $u^\star$ and $\Pi^\star$, respectively. Then, (\ref{eq:RH_relax}a) and (\ref{eq:RH_relax}b) lead to the expressions (\ref{eq:sol_PR_relax2}c) and (\ref{eq:sol_PR_relax2}d). \qed
\end{proof}

We end this section by summing up our strategy to solve the problem (\ref{eq:Euler_eq}). We split  (\ref{eq:Euler_eq}a) into acoustic and transport parts by applying the decomposition introduced in section~\ref{sec:LP_splitting}: (i)  the acoustic step is approximated with the homogeneous relaxation system (\ref{eq:acq_step_relax_homogen}); (ii) we apply the transport step (\ref{eq:trp_step}). We stress that relaxation mechanisms in the right-hand-side of (\ref{eq:acq_step_relax}) are taken into account through the initial condition (\ref{eq:Euler_eq}b) which is transformed  into Lagrange variables with data at equilibrium, \ie, ${\bf w}_0=\big(1/\rho_0,\rho u_0/\rho_0,\mathrm{p}(1/\rho_0)\big)^\top$ pointwise. We now propose to mimic this strategy at the discrete level in the next section, where Theorems~\ref{th:RP_hyp_relax} and \ref{th:RP_acq_relax} will be used to design the numerical flux for the acoustic step.

%
%
\section{Discontinuous Galerkin formulation}\label{sec:DG_discr}

The DG method consists in defining a discrete weak formulation of problem (\ref{eq:trp_step}), (\ref{eq:acq_step_relax_homogen}), and (\ref{sec:Euler_eqn}b). The domain is discretized with a uniform grid $\Omega_h=\cup_{j\in\mathbb{Z}}\kappa_j$ with cells $\kappa_j = [x_{j-1/2}, x_{j+1/2}]$, $x_{j+1/2}=(j+\tfrac{1}{2})h$ and $h>0$ the space step (see Figure \ref{fig:stencil_1D}).

\subsection{Numerical solution and Lagrange polynomials} We look for approximate solutions in the function space of discontinuous polynomials
\begin{equation}\label{eq:Vhp-space}
 {\cal V}_h^p=\{v_h\in L^2(\Omega_h):\;v_h|_{\kappa_{j}}\in{\cal P}_p(\kappa_{j}),\; \kappa_j\in\Omega_h\},
\end{equation}

\noindent where ${\cal P}_p(\kappa_{j})$ denotes the space of polynomials of degree at most $p$ in the element $\kappa_{j}$. The approximate solutions to systems (\ref{eq:trp_step}) and (\ref{eq:acq_step_relax_homogen}) are sought under the form
\begin{subeqnarray}\label{eq:num_sol}
 {\bf u}_h(x,t)&=&\sum_{l=0}^{p}\phi_j^l(x){\bf U}_j^{l}(t), \quad \forall x\in\kappa_{j},\, \kappa_j\in\Omega_h,\, t\geq0,\\
 {\bf w}_h(x,t)&=&\sum_{l=0}^{p}\phi_j^l(x){\bf W}_j^{l}(t), \quad \forall x\in\kappa_{j},\, \kappa_j\in\Omega_h,\, t\geq0,
\end{subeqnarray}

\noindent where ${\bf U}_j^l=(\rho_j^l,\rho U_j^l)^\top$ constitute the degrees of freedom (DOFs) in the element $\kappa_j$ and are associated to conservative variables, while ${\bf W}_j^l=(\tau_j^l,U_j^l,\Pi_j^l)^\top$ are coefficients associated to Lagrange variables. The subset $(\phi_j^0,\dots,\phi_j^{p})$ constitutes a basis of ${\cal V}_h^p$ restricted onto a given element. In this work we will use the Lagrange interpolation polynomials $\ell_{0\leq k\leq p}$ associated to the Gauss-Lobatto nodes over the segment $[-1,1]$: $s_0=-1<s_1<\dots<s_p=1$:
\begin{equation}\label{eq:lagrange_poly}
 \ell_k(s_l)=\delta_{k,l}, \quad 0\leq k,l \leq p,
\end{equation}

\noindent with $\delta_{k,l}$ the Kronecker symbol. The basis functions in a given element $\kappa_j$ thus write $\phi_j^k(x)=\ell_k(\sigma_j(x))$ where $\sigma_j(x)=2(x-x_j)/h$ and $x_j=(x_{j+1/2}+x_{j-1/2})/2$ denotes the center of the element.

The DOFs thus correspond to the point values of the solution, \eg given $0\leq k\leq p$, $j$ in $\mathbb{Z}$, and $t\geq0$, we have ${\bf u}_h(x_j^k,t)={\bf U}_j^k(t)$ for $x_j^k=x_j+s_kh/2$. The left and right traces of the numerical solution at interfaces $x_{j\pm1/2}$ of a given element hence read (see figure~\ref{fig:stencil_1D}):
\begin{subeqnarray}\label{eq:LR_traces}
 {\bf u}_{j+\frac{1}{2}}^-(t) &:=& {\bf u}_h(x_{j+\frac{1}{2}}^-,t) = {\bf U}_j^p(t), \quad \forall t\geq0,\\
 {\bf u}_{j-\frac{1}{2}}^+(t) &:=& {\bf u}_h(x_{j-\frac{1}{2}}^+,t) = {\bf U}_j^0(t), \quad \forall t\geq0.
\end{subeqnarray}

\begin{figure}
\begin{center}
\epsfig{figure=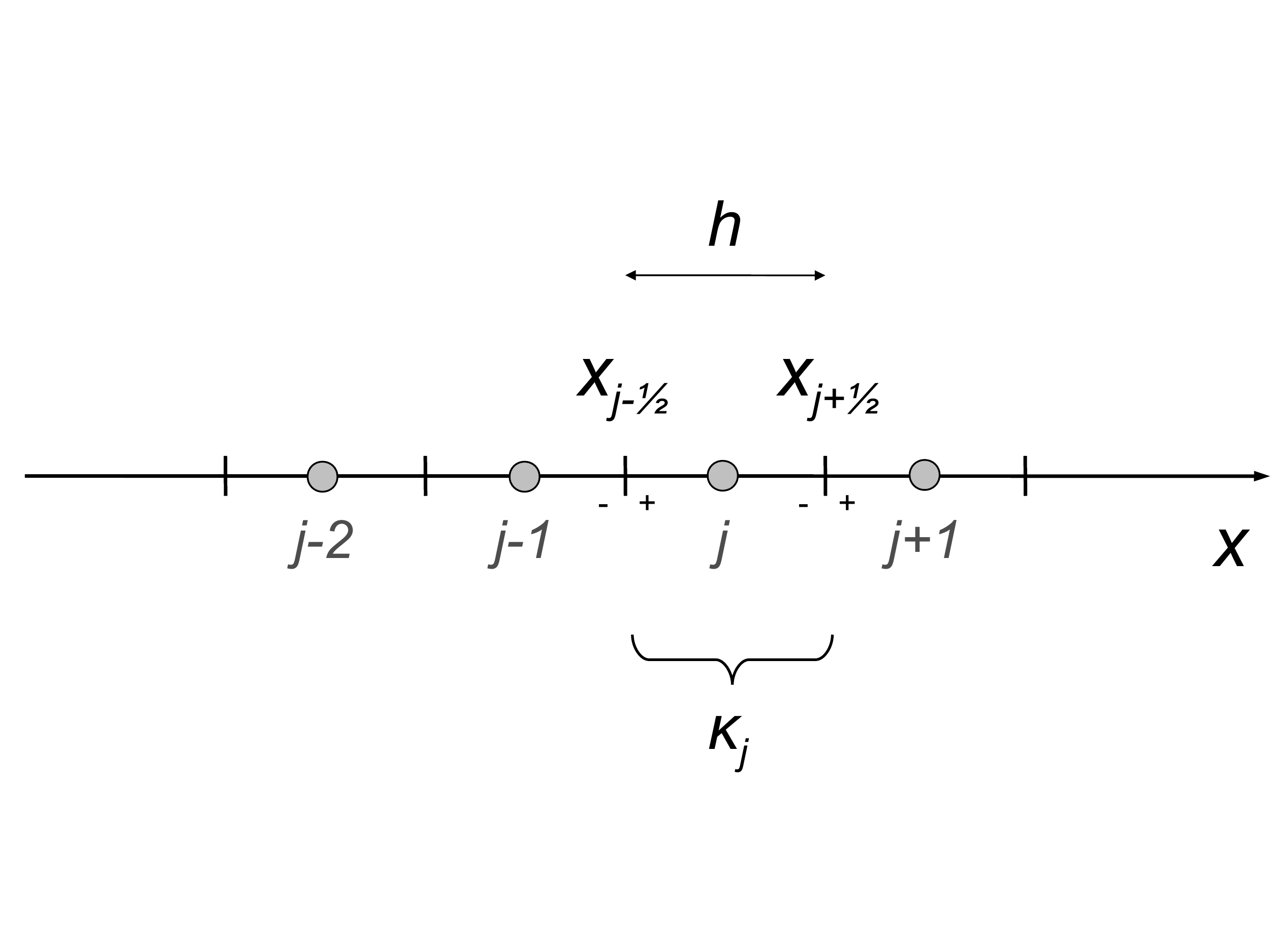,width=6cm,clip=true,trim=0cm 3.5cm 0cm 3.cm}
\caption{Mesh with definition of left and right traces at interfaces $x_{j\pm1/2}$.}
\label{fig:stencil_1D}
\end{center}
\end{figure}

Moreover, assuming equilibrium $\varepsilon\rightarrow0$ for the relaxation system (\ref{eq:acq_step_relax}), the physical and relaxation pressures satisfy
\begin{equation}\label{eq:inst_relaxation}
 \mathrm{p}\big(\tau_j^k(t)\big) = \Pi_j^k(t),\quad 0\leq k \leq p,\quad t>0.
\end{equation}

As a consequence, the conservative and Lagrange variables (\ref{eq:num_sol}) may be related in a weak sense by the relations
\begin{subeqnarray}\label{eq:weak_chge_var}
 {\bf W}_j^k(t) &=& {\bf w}({\bf U}_j^k(t))\quad 0\leq k \leq p,\quad t\geq0,\\
 {\bf U}_j^k(t) &=& {\bf u}({\bf W}_j^k(t))\quad 0\leq k \leq p,\quad t\geq0,
\end{subeqnarray}

\noindent where 
\begin{subeqnarray*}
 {\bf w}:\Omega^a\rightarrow\Omega^r&;&{\bf u}\mapsto{\bf w}({\bf u})=\big(\tfrac{1}{\rho}, \tfrac{\rho u}{\rho}, \mathrm{p}\big(\tfrac{1}{\rho}\big)\big)^\top, \\
 {\bf u}:\Omega^r\rightarrow\Omega^a&;&{\bf w}\mapsto{\bf u}({\bf w})=\big(\tfrac{1}{\tau}, \tfrac{u}{\tau}\big)^\top
\end{subeqnarray*}

\noindent denote, with a slight abuse, the change from conservative to Lagrange variables and its inverse.

\subsection{Space discretization} 

We first consider the space discretization of the relaxation approximation (\ref{eq:acq_step_relax_homogen}) of the acoustic step without source term, \ie, equation (\ref{eq:acq_step_relax_homogen}). Substitute (\ref{eq:num_sol}b) in equation (\ref{eq:acq_step_relax_homogen}) with $\varepsilon\rightarrow\infty$, multiply it with a test function $v_h$ in ${\cal V}_h^p$ with support in a given element $\kappa_j$ and integrate by parts over $\kappa_j$ to obtain
\begin{equation}
 \int_{\kappa_j} v_h\partial_t{\bf w}_hdx - \int_{\kappa_j} {\bf f}_a({\bf w}_h)\partial_x(v_h\tau_h)dx + \Big[v_h\tau_h{\bf h}_a({\bf w}_h^-,{\bf w}_h^+)\Big]_{x_{j-\frac{1}{2}}}^{x_{j+\frac{1}{2}}} = 0,
\end{equation}

\noindent after replacing the physical flux at interface by the numerical flux ${\bf h}_a:\Omega^r\times\Omega^r\rightarrow\mathbb{R}^3$ defined from the solution of the Riemann problem (\ref{eq:sol_PR_relax}):
\begin{equation}\label{eq:num_flux_acou}
 {\bf h}_a({\bf w}_{j+\frac{1}{2}}^-,{\bf w}_{j+\frac{1}{2}}^+) = {\bf f}_a\big({\boldsymbol{\cal W}}(0;{\bf w}_{j+\frac{1}{2}}^-,{\bf w}_{j+\frac{1}{2}}^+)\big) = \left(
\begin{array}{c}
  -u_{j+\frac{1}{2}}^\star\\
  \Pi_{j+\frac{1}{2}}^\star\\
  a^2u_{j+\frac{1}{2}}^\star
\end{array}
\right),
\end{equation}

\noindent with data at equilibrium, \ie, $\Pi_{j+\frac{1}{2}}^\pm=\mathrm{p}(\tau_{j+\frac{1}{2}}^\pm)$. Note that there is no ambiguity in defining the numerical flux at $\xi=0$ since the application $\xi\mapsto {\bf f}_a\big({\boldsymbol{\cal W}}(\xi;{\bf w}_L,{\bf w}_R)\big)$ is continuous at the origin because of the Rankine-Hugoniot relations associated with the steady wave.

Using a second integration by parts, the semi-discrete equation may be equivalently written as
\begin{equation}\label{eq:semi-discrete_eq_acou}
 \int_{\kappa_j} v_h\partial_t{\bf w}_hdx + \int_{\kappa_j} v_h\tau_h\partial_x{\bf f}_a({\bf w}_h)dx + \Big[v_h\tau_h\big({\bf h}_a({\bf w}_h^-,{\bf w}_h^+)-{\bf f}_a({\bf w}_h)\big)\Big]_{x_{j-\frac{1}{2}}}^{x_{j+\frac{1}{2}}} = 0.
\end{equation}

We now introduce the space discretization of the transport problem (\ref{eq:trp_step}) where we consider the approximate solution (\ref{eq:num_sol}a). Using a similar approach as for the acoustic step, with one integration by parts, the semi-discrete form of the DG discretization in space reads
\begin{equation}\label{eq:semi-discrete_eq_trp}
 \int_{\kappa_j} v_h\partial_t{\bf u}_hdx + \int_{\kappa_j} (v_hu_h)\partial_x{\bf u}_hdx + \Big[v_h\big({\bf h}_t({\bf u}_h^-,{\bf u}_h^+)-h_u(u_h^-,u_h^+){\bf u}_h\big)\Big]_{x_{j-\frac{1}{2}}}^{x_{j+\frac{1}{2}}} = 0,
\end{equation}

\noindent where ${\bf h}_t:\Omega^a\times\Omega^a\rightarrow\mathbb{R}^2$ and $h_u:\mathbb{R}\times\mathbb{R}\rightarrow\mathbb{R}$ denote Lipschitz-continuous numerical fluxes consistent with the physical fluxes of the transport step: ${\bf h}_t({\bf u},{\bf u})=u{\bf u}$ and $h_u(u,u)=u$. In this work, we use upwind numerical fluxes. Introducing
\begin{equation}\label{eq:upwind_flux}
\hat{\bf u}_{j+\frac{1}{2}} = \left\{
\begin{array}{ll}
  {\bf U}_j^{p}, & u_{j+\frac{1}{2}}^\star > 0,\\
  {\bf U}_{j+1}^{0}, & u_{j+\frac{1}{2}}^\star \leq 0,
\end{array}
\right.
\end{equation}

\noindent where the quantity $u_{j+1/2}^\star$ is defined from the first component of the numerical flux (\ref{eq:num_flux_acou}) for the acoustic step, we set
\begin{equation}\label{eq:flux_num_trp_ft}
 {\bf h}_t({\bf u}_{j+\frac{1}{2}}^-,{\bf u}_{j+\frac{1}{2}}^+) = u_{j+\frac{1}{2}}^\star \hat{\bf u}_{j+\frac{1}{2}} = (u_{j+\frac{1}{2}}^\star)^+{\bf U}_j^{p}+(u_{j+\frac{1}{2}}^\star)^-{\bf U}_{j+1}^{0},
\end{equation}

\noindent where 
\begin{equation*}
 (u_{j+\frac{1}{2}}^\star)^+=\max(u_{j+\frac{1}{2}}^\star,0), \quad  (u_{j+\frac{1}{2}}^\star)^-=\min(u_{j+\frac{1}{2}}^\star,0)
\end{equation*}

\noindent represent positive and negative parts of $u_{j+\frac{1}{2}}^\star$. Then, we set
\begin{equation}\label{eq:flux_num_trp_u}
 h_u(u_{j+\frac{1}{2}}^-,u_{j+\frac{1}{2}}^+) = u_{j+\frac{1}{2}}^\star.
\end{equation}

The integrals in the semi-discrete equations are approximated by using a numerical quadrature. We use a Gauss-Lobatto quadrature rule with nodes associated with the interpolation points of the numerical solution
\begin{equation*}
 \int_{\kappa_j}f(x)dx \simeq \frac{h}{2} \sum_{l=0}^p \omega_l f(x_j^l),
\end{equation*}

\noindent with $\omega_l>0$, $\sum_{l=0}^p\omega_l=2$, $x_j^l=x_j+s_lh/2$ the weights and nodes of the quadrature rule, and $s_l$ defined in (\ref{eq:lagrange_poly}). For $p+1$ integration points, this quadrature is exact for polynomials of degree $\deg(f)\leq2p-1$. It is worth noting that the integration by parts used in equations (\ref{eq:semi-discrete_eq_acou}) and (\ref{eq:semi-discrete_eq_trp}) are replaced by summation by parts of the form
\begin{equation*}
 \langle f,d_xv_h\rangle_j^p+\langle d_xf,v_h\rangle_j^p=f(x_j^p)v_h(x_j^p)-f(x_j^0)v_h(x_j^0), \quad \forall v_h\in{\cal V}_h^p,
\end{equation*}

\noindent where 
\begin{equation*}
 \langle f,g\rangle_j^p = \frac{h}{2}\sum_{l=0}^p\omega_lf(x_j^l)g(x_j^l)
\end{equation*}

\noindent represents the discrete inner product in the element $\kappa_j$. As noticed in \cite{kopriva_gassner10}, this operation holds true in a weak sense for general functions $f$ by considering the interpolation polynomials of degree $p$ of the function, let say $f_h$, at integration points, $x_j^{0\leq k\leq p}$.

%
%
\section{Time discretization}\label{sec:time_discr}

In section~\ref{sec:1st_order_time_discr}, we introduce the fully discrete scheme for a one-step first-order implicit-explicit time discretization and analyze its properties in section~\ref{sec:ana_1st_order_time_discr}, while we briefly discuss high-order time discretization in section~\ref{sec:HO_time_discr} and limiting strategy in section~\ref{sec:limiters}.

\subsection{First order time discretization}\label{sec:1st_order_time_discr}

Let $t^{(n)}=n\Delta t$ with $\Delta t>0$ the time step, the time integration of the Euler equations with Lagrange-projection over a time step is done with a two-step sequential splitting: (i) homogeneous acoustic step (\ref{eq:semi-discrete_eq_acou}) over $(t^{(n)},t^{(n+1^-)}]$ and (ii) transport step (\ref{eq:semi-discrete_eq_trp}) over $(t^{(n+1^-)},t^{(n+1)}]$. Again, relaxation mechanisms in (\ref{eq:acq_step_relax}) are taken into account by imposing data at equilibrium at each time step: $\Pi_h^{(n)}=\mathrm{p}(\tau_h^{(n)})$.

We use an implicit backward-Euler scheme for the time discretization of the acoustic step over $(t^{(n)},t^{(n+1^-)}]$. Approximating the mass variable by $\Delta m=\tau_j^{k}(t^{(n)})h$ at point $x_j^k$ over $(t^{(n)},t^{(n+1^-)}]$ and setting $v_h=\phi_j^k$ into (\ref{eq:semi-discrete_eq_acou}), the discrete scheme now reads
\begin{eqnarray}\label{eq:discrete_eq_acou}
 {\bf W}_j^{k,n+1^-} = {\bf W}_j^{k,n} &-& \lambda_k\tau_j^{k,n}\Big[\big\langle\partial_x{\bf f}_a({\bf w}_h^{n+1^-}),\phi_j^k\big\rangle_j^p \nonumber\\
&+& \delta_{k,p}\big({\bf h}_a({\bf W}_{j}^{p,n+1^-},{\bf W}_{j+1}^{0,n+1^-})-{\bf f}_a({\bf W}_j^{p,n+1^-})\big) \nonumber\\
&-& \delta_{k,0}\big({\bf h}_a({\bf W}_{j-1}^{p,n+1^-},{\bf W}_{j}^{0,n+1^-})-{\bf f}_a({\bf W}_j^{0,n+1^-})\big)\Big],
\end{eqnarray}

\noindent for all $0\leq k\leq p$ and $j$ in $\mathbb{Z}$, where we have used the notations ${\bf W}_j^{k,n}={\bf W}_j^{k}(t^{(n)})$ and $\lambda_k=2\lambda/\omega_k$ with $\lambda=\Delta t/h$. For the sake of clarity, we now use the short notations for the components of the Riemann solver 
\begin{equation*}
 {\bf h}_a({\bf W}_{j}^{p,n+1^-},{\bf W}_{j+1}^{0,n+1^-}) = \left(
\begin{array}{c}
  -u_{j+\frac{1}{2}}^\star\\
  \Pi_{j+\frac{1}{2}}^\star\\
  a^2u_{j+\frac{1}{2}}^\star
\end{array}
\right)
\end{equation*}

\noindent wihtout any possible confusion on the time and trace values.

The subcharacteristic condition (\ref{eq:sub_cond}) is imposed at the discrete level by requiring that
\begin{equation}\label{eq:sub_cond_discr}
a > \max_{j\in\mathbb{Z}}\max_{0\leq k\leq p}\max_{\theta\in[0,1]}\sqrt{-\mathrm{p}'\big(\theta\tau_j^{k,n}+(1-\theta)\tau_j^{k,n+1^-}\big)}, \quad n\in\mathbb{N}.
\end{equation}

Likewise, we use an explicit forward Euler time integration of the semi-discrete transport equations (\ref{eq:semi-discrete_eq_trp}). Introducing the definitions of the numerical fluxes (\ref{eq:flux_num_trp_ft}) and (\ref{eq:flux_num_trp_u}), we obtain
\begin{eqnarray}\label{eq:discrete_eq_trp}
 {\bf U}_j^{k,n+1} &=& {\bf U}_j^{k,n+1^-} - \lambda_k\Big[\big\langle u_h^{n+1^-}\partial_x{\bf u}_h^{n+1^-},\phi_j^k\big\rangle_j^p \nonumber\\
&+&\!  \delta_{k,p}u_{j+\frac{1}{2}}^\star\big(\hat{\bf u}_{j+\frac{1}{2}}^{n+1^-}-{\bf U}_j^{p,n+1^-}) \!-\! \delta_{k,0}u_{j-\frac{1}{2}}^\star\big(\hat{\bf u}_{j-\frac{1}{2}}^{n+1^-}-{\bf U}_j^{0,n+1^-})\Big].
\end{eqnarray}

The discrete problem for the Euler equations now reads: for all time $t^{(n+1)}$ with $n\geq0$ find ${\bf u}_h(\cdot,t^{(n+1)})$ in $({\cal V}_h^p)^2$ such that equations (\ref{eq:discrete_eq_acou}) and (\ref{eq:discrete_eq_trp}) are satisfied with 
\begin{subeqnarray}\label{eq:projections_n_n+1^-}
 {\bf W}_j^{k,n}&=&{\bf w}\big({\bf U}_j^{k,n}\big), \quad 0\leq k\leq p, j\in\mathbb{Z},\\
 {\bf U}_j^{k,n+1^-}&=&{\bf u}\big({\bf W}_j^{k,n+1^-}\big), \quad 0\leq k\leq p, j\in\mathbb{Z},
\end{subeqnarray}

\noindent given by transformations (\ref{eq:weak_chge_var}).

The time sequence ${\bf u}_h(\cdot,t^{(n)})$ is associated with the initial condition
\begin{equation}\label{eq:discr_CI}
 \int_{\Omega_h} v_h {\bf w}_h(x,0) dx = \int_{\Omega_h} v_h {\bf w}\big({\bf u}_0(x)\big) dx, \quad \forall v_h\in{\cal V}_h^p,
\end{equation}

\noindent which reduces to ${\bf W}_j^k(0)={\bf w}\big({\bf u}_0(x_j^k)\big)$ for all $0\leq k\leq p$ and $j$ in $\mathbb{Z}$.

The following theorem allows us to give the final form (\ref{eq:discrete_eq_Euler}) of the LPDG scheme.

\begin{theorem}[Consistency and conservation]
 The discrete schemes (\ref{eq:discrete_eq_acou}) and (\ref{eq:discrete_eq_trp}) with projections (\ref{eq:projections_n_n+1^-}) constitute a conservative approximation consistent in time and space with the Euler equations (\ref{eq:Euler_eq}a) of the form
\begin{equation}\label{eq:discrete_eq_Euler}
 {\bf U}_j^{k,n+1} = {\bf U}_j^{k,n} - \lambda_k\Big[-\big\langle{\bf f}({\bf u}_h^{n+1^-}),d_x\phi_j^k\big\rangle_j^p + \delta_{k,p}{\bf h}_{j+\frac{1}{2}}^{n+1^-}-\delta_{k,0}{\bf h}_{j-\frac{1}{2}}^{n+1^-}\Big],
\end{equation}

\noindent with 
\begin{equation}
 {\bf h}_{j+\frac{1}{2}}^{n+1^-} = \left(
\begin{array}{ll}
  {u}_{j+\frac{1}{2}}^\star\hat{\rho}_{j+\frac{1}{2}}^{n+1^-}\\
  {u}_{j+\frac{1}{2}}^\star\hat{\rho u}_{j+\frac{1}{2}}^{n+1^-} + \Pi_{j+\frac{1}{2}}^\star
\end{array}
\right),
\end{equation}

\noindent evaluated from DOFs at time $t^{(n+1^-)}$.
\end{theorem}

\begin{proof}
We first observe that the first component of the discrete equation (\ref{eq:discrete_eq_acou}) reads $\tau_j^{k,n+1^-} = L_j^{k,n+1^-}\tau_j^{k,n}$ with
\begin{subeqnarray}\label{eq:def_Ljkn}
 L_j^{k,n+1^-} &=& 1+\lambda_{k}\Big(\big\langle\partial_xu_h^{n+1^-},\phi_j^k\big\rangle_j^p \nonumber\\
 && +\delta_{k,p}(u_{j+\frac{1}{2}}^\star-U_j^{p,n+1^-})-\delta_{k,0}(u_{j-\frac{1}{2}}^\star-U_j^{0,n+1^-})\Big) \\
 &=& 1+\lambda_{k}\Big(-\big\langle u_h^{n+1^-},d_x\phi_j^k\big\rangle_j^p+\delta_{k,p}u_{j+\frac{1}{2}}^\star-\delta_{k,0}u_{j-\frac{1}{2}}^\star\Big).
\end{subeqnarray}

Using the change of variables (\ref{eq:weak_chge_var}b) and assuming that $\tau_j^{k,n}>0$ and $\tau_j^{k,n+1^-}>0$ (see equation (\ref{eq:pos_ac_step}) in Lemma~\ref{th:pos_ineq_acq}), the two first components of (\ref{eq:discrete_eq_acou}) may thus be rewritten under the form
\begin{subeqnarray}\label{eq:discrete_ac_step_in_cons_form}
 L_j^{k,n+1^-}\rho_j^{k,n+1^-} &=& \rho_j^{k,n},\\
 L_j^{k,n+1^-}\rho U_j^{k,n+1^-} &=& \rho U_j^{k,n} - \lambda_{k}\Big(\big\langle\partial_x\Pi_h^{n+1^-},\phi_j^k\big\rangle_j^p \nonumber\\
 &+& \delta_{k,p}\big(\Pi_{j+\frac{1}{2}}^\star\!-\!\Pi_j^{p,n+1^-}\big)-\delta_{k,0}\big(\Pi_{j-\frac{1}{2}}^\star\!-\!\Pi_j^{0,n+1^-}\big) \Big).
\end{subeqnarray}

Using the expressions of the  numerical fluxes (\ref{eq:flux_num_trp_ft}) and (\ref{eq:flux_num_trp_u}) with ${\bf f}_t({\bf U}_j^{k,n})=U_j^{k,n}{\bf U}_j^{k,n}$, the discrete transport step (\ref{eq:discrete_eq_trp}) may be rewritten as
\begin{eqnarray*}
 {\bf U}_j^{k,n+1} &=& {\bf U}_j^{k,n+1^-} - \lambda_k\Big[\big\langle\partial_x(u_h^{n+1^-}{\bf u}_h^{n+1^-}),\phi_j^k\big\rangle_j^p-\big\langle\partial_xu_h^{n+1^-}{\bf u}_h^{n+1^-},\phi_j^k\big\rangle_j^p \\
&&+ \delta_{k,p}\big({\bf h}_t({\bf u}_{j+\frac{1}{2}}^-,{\bf u}_{j+\frac{1}{2}}^+)-{\bf f}_t({\bf U}_j^{p,n+1^-})+{\bf f}_t({\bf U}_j^{p,n+1^-})-u_{j+\frac{1}{2}}^\star{\bf U}_j^{p,n+1^-}\big) \\
&&- \delta_{k,0}\big({\bf h}_t({\bf u}_{j-\frac{1}{2}}^-,{\bf u}_{j-\frac{1}{2}}^+)-{\bf f}_t({\bf U}_j^{0,n+1^-})+{\bf f}_t({\bf U}_j^{0,n+1^-})-u_{j-\frac{1}{2}}^\star{\bf U}_j^{0,n+1^-}\big)\Big]\\
&=& L_j^{k,n+1^-}{\bf U}_j^{k,n+1^-} - \lambda_k\Big[\big\langle\partial_x(u_h^{n+1^-}{\bf u}_h^{n+1^-}),\phi_j^k\big\rangle_j^p \\
&&+ \delta_{k,p}\big({\bf h}_t({\bf u}_{j+\frac{1}{2}}^-,{\bf u}_{j+\frac{1}{2}}^+)-{\bf f}_t({\bf U}_j^{p,n+1^-})\big) \\
&&- \delta_{k,0}\big({\bf h}_t({\bf u}_{j-\frac{1}{2}}^-,{\bf u}_{j-\frac{1}{2}}^+)-{\bf f}_t({\bf U}_j^{0,n+1^-})\big)\Big].
\end{eqnarray*}

Summing the two above schemes, we get
\begin{eqnarray*}
 {\bf U}_j^{k,n+1} &=& {\bf U}_j^{k,n} - \lambda_k\Big[\big\langle\partial_x{\bf f}_E({\bf u}_h^{n+1^-},{\bf w}_h^{n+1^-}),\phi_j^k\big\rangle_j^p \\
&& + \delta_{k,p}\big({\bf h}_{j+\frac{1}{2}}^{n+1^-}-{\bf f}_E({\bf U}_j^{p,n+1^-},{\bf W}_j^{p,n+1^-})\big) \\
&& - \delta_{k,0}\big({\bf h}_{j-\frac{1}{2}}^{n+1^-}-{\bf f}_E({\bf U}_j^{0,n+1^-},{\bf W}_j^{0,n+1^-})\big)\Big],
\end{eqnarray*}

\noindent with ${\bf f}_E({\bf u},{\bf w})=u{\bf u}+(0,\Pi)^\top$ consistent with the Euler fluxes ${\bf f}({\bf u})$ when $\Pi=\mathrm{p}(\tau)$. This is done by imposing instantaneous relaxation (\ref{eq:inst_relaxation}), $\mathrm{p}\big(\tau_j^{k,n+1}\big)=\Pi_j^{k,n+1}$. Now, we use the equivalence between the discrete DG formulations with either one or two integration by parts when using a collocated Gaus-Lobatto quadrature \cite{kopriva_gassner10}. Applying integration by parts, the above scheme reduces to (\ref{eq:discrete_eq_Euler}) and constitutes a conservative implicit-explicit discretization of the Euler equations (\ref{eq:Euler_eq}) consistent in space and time. \qed
\end{proof}

Then, the numerical scheme (\ref{eq:discrete_eq_acou}) provides the following discrete versions of the conservation equations (\ref{eq:Riemann_invar_acou}) for the characteristic variables. Note that from the definition of the numerical flux (\ref{eq:num_flux_acou}) with (\ref{eq:sol_PR_relax2}) we have
\begin{subeqnarray}\label{eq:rel_pi_u_star_traces}
 \Pi_{j+\frac{1}{2}}^\star-\Pi_j^{p,n+1^-} &=&-a(u_{j+\frac{1}{2}}^\star-U_j^{p,n+1^-}), \\
 \Pi_{j-\frac{1}{2}}^\star-\Pi_j^{0,n+1^-} &=& a(u_{j-\frac{1}{2}}^\star-U_j^{0,n+1^-}).
\end{subeqnarray}

Then, setting 
\begin{equation}\label{eq:def_vec_cev_j}
 \vecloc{W}_j^{k,n}=\Pi_j^{k,n}+aU_j^{k,n}, \quad J_j^{k,n}=\Pi_j^{k,n}+a^2\tau_j^{k,n}, \quad \cevloc{W}_j^{k,n}=\Pi_j^{k,n}-aU_j^{k,n},
\end{equation}

\noindent and using relations (\ref{eq:rel_pi_u_star_traces}), the three operations (\ref{eq:discrete_eq_acou}c)$+a$(\ref{eq:discrete_eq_acou}b), (\ref{eq:discrete_eq_acou}c)$+a^2$(\ref{eq:discrete_eq_acou}a) and (\ref{eq:discrete_eq_acou}c)$-a$(\ref{eq:discrete_eq_acou}b) give respectively:
\begin{subeqnarray}\label{eq:riemann_inv_eqn_acou}
\vecloc{W}_j^{k,n+1^-} \hspace{-0.15cm} &=& \hspace{-0.05cm} \vecloc{W}_j^{k,n} \!\! - a\lambda_k\tau_j^{k,n}\! \Big[\!\big\langle\partial_x\vecloc{w}_h^{n+1^-}\!\! ,\! \phi_j^k\big\rangle_j^p \! -\! \delta_{k,0}\! \big(\vecloc{W}_{j-1}^{p,n+1^-} \hspace{-0.2cm}- \hspace{-0.1cm}\vecloc{W}_{j}^{0,n+1^-}\big)\! \Big]\! , \\
J_j^{k,n+1^-} \hspace{-0.15cm} &=& \hspace{-0.05cm} J_j^{k,n}, \\
\cevloc{W}_j^{k,n+1^-} \hspace{-0.15cm} &=& \hspace{-0.05cm} \cevloc{W}_j^{k,n} \!\! + a\lambda_k\tau_j^{k,n}\! \Big[\!\big\langle\partial_x\cevloc{w}_h^{n+1^-}\!\! ,\! \phi_j^k\big\rangle_j^p \! +\! \delta_{k,p}\! \big(\cevloc{W}_{j+1}^{0,n+1^-} \hspace{-0.2cm}- \hspace{-0.1cm}\cevloc{W}_{j}^{p,n+1^-}\big)\! \Big]\! .
\end{subeqnarray}

Equation (\ref{eq:discrete_eq_Euler}) constitutes the LPDG scheme where the state ${\bf u}_h^{n+1^-}$ is evaluated from the linear implicit system (\ref{eq:riemann_inv_eqn_acou}a,c) that may be easily solved for the discrete characteristic variables (\ref{eq:def_vec_cev_j}), $\tau_j^{k,n+1^-}$ being then given explicitly by (\ref{eq:riemann_inv_eqn_acou}b). Then, the Lagrange variables are obtained by inverting relations (\ref{eq:def_vec_cev_j}) at time $t^{(n+1^-)}$. This result is essential for the performances of the present method.

%
\subsection{Properties of the discrete scheme}\label{sec:ana_1st_order_time_discr}

In this section, we discuss the properties of the LPDG scheme with first-order time integration and arbitrary order for the space discretization. The main results are given in Theorem~\ref{th:pos_ineq_LPDG} and prove positivity and entropy inequality for the mean value of the numerical solution
\begin{equation}
 \ol{\bf u}_j^{n} := \frac{1}{h}\int_{\kappa_j}{\bf u}_h(x,t^{(n)})dx=\sum_{l=0}^p\frac{\omega_k}{2}{\bf U}_j^{l,n}.
\end{equation}

The entropy inequality applies to the total energy that we introduce at the discrete level via its interpolant
\begin{equation}\label{eq:discr_tot_nrj}
 \rho E_h(x,t) = \sum_{l=0}^p \phi_j^l(x)\rho E_j^l(t), \quad \forall x\in\kappa_j, t\geq0,
\end{equation}

\noindent with $\rho E_j^l(t)=\rho E\big({\bf U}_j^l(t)\big)$.
\begin{lemma}
\label{th:pos_ineq_acq}
Assume that $\rho_{j\in\mathbb{Z}}^{0\leq k\leq p,n}>0$, then under the CFL condition
\begin{equation}\label{eq:CFL_LPDG}
 \lambda\max_{j\in\mathbb{Z}}\max_{0\leq k\leq p}\frac{1}{\omega_k}\Big(\big\langle u_h^{n+1^-},d_x\phi_j^k\big\rangle_j^p-\delta_{k,p}(u_{j+\frac{1}{2}}^\star)^-+\delta_{k,0}(u_{j-\frac{1}{2}}^\star)^+\Big) < \frac{1}{2},
\end{equation}
\noindent we have 
\begin{equation}\label{eq:pos_ac_step}
\rho_j^{k,n+1^-}>0
\end{equation}

\noindent and 
\begin{eqnarray}\label{eq:convex_comb_trp}
 \ol{\bf u}_j^{n+1} &=& \sum_{k=0}^p\bigg(\frac{\omega_k}{2}-\lambda\Big(\big\langle u_h^{n+1^-},d_x\phi_j^k\big\rangle_j^p-\delta_{k,p}(u_{j+\frac{1}{2}}^\star)^-+\delta_{k,0}(u_{j-\frac{1}{2}}^\star)^+\Big)\bigg){\bf U}_j^{k,n+1^-} \nonumber \\ 
 && -\lambda (u_{j+\frac{1}{2}}^\star)^-{\bf U}_{j+1}^{0,n+1^-} + \lambda (u_{j-\frac{1}{2}}^\star)^+{\bf U}_{j-1}^{p,n+1^-}
\end{eqnarray}
\noindent is a convex combination of DOFs at time $t^{(n+1^-)}$.
\end{lemma}

\begin{proof}
 From (\ref{eq:discrete_ac_step_in_cons_form}a), we have $L_j^{k,n+1^-}\rho_j^{k,n+1^-} = \rho_j^{k,n}$ with $L_j^{k,n+1^-}>0$ from (\ref{eq:def_Ljkn}b) and condition (\ref{eq:CFL_LPDG}), hence $\rho_j^{k,n+1^-}>0$. Then using (\ref{eq:discrete_eq_trp}) to evaluate the mean value at time $t^{(n+1)}$ together with definition of the upwind flux (\ref{eq:upwind_flux}), we get
\begin{eqnarray*}
 \ol{\bf u}_j^{n+1} &=& \sum_{k=0}^p\frac{\omega_k}{2}{\bf U}_j^{k,n+1} \\
 &=& \sum_{k=0}^p\frac{\omega_k}{2}\bigg( {\bf U}_j^{k,n+1^-} - \lambda_k\Big[\big\langle u_h^{n+1^-}\partial_x{\bf u}_h^{n+1^-},\phi_j^k\big\rangle_j^p \nonumber\\
&& + \delta_{k,p}(u_{j+\frac{1}{2}}^\star)^-\big({\bf U}_{j+1}^{0,n+1^-}-{\bf U}_j^{p,n+1^-}) \\
&& - \delta_{k,0}(u_{j-\frac{1}{2}}^\star)^+\big({\bf U}_{j-1}^{p,n+1^-}-{\bf U}_j^{0,n+1^-})\Big] \bigg).
\end{eqnarray*}

Then, developing the expression of the volume integral, the above equation reads
\begin{eqnarray*}
 \ol{\bf u}_j^{n+1}  &=& \sum_{k=0}^p\frac{\omega_k}{2}{\bf U}_j^{k,n+1^-} - \lambda\sum_{k=0}^p\bigg(\frac{\omega_kh}{2}U_j^{k,n+1^-}\sum_{l=0}^pd_x\phi_j^l(x_j^k){\bf U}_j^{l,n+1^-} \\
&+& \delta_{k,p}(u_{j+\frac{1}{2}}^\star)^-({\bf U}_{j+1}^{0,n+1^-}-{\bf U}_{j}^{p,n+1^-}) - \delta_{k,0}(u_{j-\frac{1}{2}}^\star)^+({\bf U}_{j-1}^{p,n+1^-}-{\bf U}_{j}^{0,n+1^-}) \bigg).
\end{eqnarray*}
Inverting indices in the double sum and rearranging terms, one easily obtains (\ref{eq:convex_comb_trp}) and positivity of coefficients follows from (\ref{eq:CFL_LPDG}). Finally, note that all  coefficients in (\ref{eq:convex_comb_trp}) are positive from condition (\ref{eq:CFL_LPDG}) with unit sum from $\sum_{k=0}^pd_x\phi_j^k=0$. Thus (\ref{eq:convex_comb_trp}) is a convex combination. \qed
\end{proof}

\begin{lemma}
Assume that $\rho_{j\in\mathbb{Z}}^{0\leq k\leq p,n}>0$, then the discrete acoustic step satisfies the following discrete entropy inequality
\begin{equation}\label{eq:entropy_ineq_acq_step}
 \eta_j^{k,n+1^-} - \eta_j^{k,n} - 2a^2\lambda_k\tau_j^{k,n}\Big[\big\langle\Pi_h^{n+1^-}u_h^{n+1^-},d_x\phi_j^k\big\rangle_j^p - \delta_{k,p}H_{j+\frac{1}{2}}^{n+1^-} + \delta_{k,0}H_{j-\frac{1}{2}}^{n+1^-}\Big] \leq 0,
\end{equation}
\noindent with
\begin{subeqnarray}\label{eq:entropy_flux_acq_step}
 \eta_j^{k,n} &=& \frac{(\vecloc{W}_j^{k,n})^2+(\cevloc{W}_j^{k,n})^2}{2}=(\Pi_j^{k,n})^2+a^2(U_j^{k,n})^2, \\
 H_{j+\frac{1}{2}}^{n+1^-} &=& \frac{(\vecloc{W}_j^{p,n+1^-})^2-(\cevloc{W}_{j+1}^{0,n+1^-})^2}{4a} = \Pi_{j+\frac{1}{2}}^\star u_{j+\frac{1}{2}}^\star.
\end{subeqnarray}
\end{lemma}

\begin{proof}
 Relations (\ref{eq:entropy_flux_acq_step}) follow directly from the definitions in (\ref{eq:def_vec_cev_j}) and (\ref{eq:sol_PR_relax2}). Then, multiplying equation (\ref{eq:riemann_inv_eqn_acou}a) with $\vecloc{W}_j^{k,n+1^-}$ gives
\begin{eqnarray*}
 \vecloc{W}_j^{k,n+1^-}(\vecloc{W}_j^{k,n+1^-} - \vecloc{W}_j^{k,n}) &+& a\lambda_k\tau_j^{k,n}\Big[\frac{\omega_kh}{2}\vecloc{W}_j^{k,n+1^-}\frac{\partial\vecloc{w}_h}{\partial x}\Big|_{x_j^k}^{n+1^-} \\ &&-\delta_{k,0}\vecloc{W}_j^{0,n+1^-}\big(\vecloc{W}_{j-1}^{p,n+1^-}-\vecloc{W}_{j}^{0,n+1^-}\big)\Big] = 0,
\end{eqnarray*}

\noindent hence
\begin{eqnarray*}
 \frac{(\vecloc{W}_j^{k,n+1^-})^2}{2} &-& \frac{(\vecloc{W}_j^{k,n})^2}{2} + a\lambda_k\tau_j^{k,n}\bigg[\big\langle\partial_x(\tfrac{(\vecloc{w}_h^{n+1^-})^2}{2}),\phi_j^k\big\rangle_j^p \\
&-& \delta_{k,0}\Big(\frac{(\vecloc{W}_{j-1}^{p,n+1^-})^2}{2} 
 -\frac{(\vecloc{W}_{j}^{0,n+1^-})^2}{2}\Big)\bigg] \\
&=& -\frac{(\vecloc{W}_j^{k,n+1^-}-\vecloc{W}_j^{k,n})^2}{2}  - a\lambda_k\tau_j^{k,n}\delta_{k,0}\frac{(\vecloc{W}_{j-1}^{p,n+1^-}-\vecloc{W}_{j}^{0,n+1^-})^2}{2}.
\end{eqnarray*}

Using integration by parts in the above equation and considering the sign of its right-hand-side, one deduces
\begin{eqnarray*}
 \frac{(\vecloc{W}_j^{k,n+1^-})^2}{2} - \frac{(\vecloc{W}_j^{k,n})^2}{2} &+& a\lambda_k\tau_j^{k,n}\bigg[-\big\langle\tfrac{(\vecloc{w}_h^{n+1^-})^2}{2},d_x\phi_j^k\big\rangle_j^p \\ &+& \delta_{k,p}\frac{(\vecloc{W}_{j}^{p,n+1^-})^2}{2} -\delta_{k,0}\frac{(\vecloc{W}_{j-1}^{p,n+1^-})^2}{2} \bigg] \leq 0.
\end{eqnarray*}

Likewise, multiplying equation (\ref{eq:riemann_inv_eqn_acou}c) with $\cevloc{W}_j^{k,n+1^-}$ and applying similar manipulations give
\begin{eqnarray*}
 \frac{(\cevloc{W}_j^{k,n+1^-})^2}{2} - \frac{(\cevloc{W}_j^{k,n})^2}{2} &-& a\lambda_k\tau_j^{k,n}\bigg[-\big\langle\tfrac{(\cevloc{w}_h^{n+1^-})^2}{2},d_x\phi_j^k\big\rangle_j^p \\  &+& \delta_{k,p}\frac{(\cevloc{W}_{j+1}^{0,n+1^-})^2}{2} -\delta_{k,0}\frac{(\cevloc{W}_{j}^{0,n+1^-})^2}{2} \bigg] \leq 0.
\end{eqnarray*}

Summing the two last equations give the desired inequality (\ref{eq:entropy_ineq_acq_step}). \qed
\end{proof}

\begin{lemma}\label{th:entropy_ineq_ac_step}
 Assume that $\rho_{j\in\mathbb{Z}}^{0\leq k\leq p,n}>0$, then under the CFL condition (\ref{eq:CFL_LPDG}) and subcharacteristic condition (\ref{eq:sub_cond_discr}), the discrete acoustic step satisfies the following discrete entropy inequality
\begin{equation}\label{eq:energy_ineq_acq_step}
E_j^{k,n+1^-} - E_j^{k,n} - \lambda_k\tau_j^{k,n}\Big[\big\langle\Pi_h^{n+1^-}u_h^{n+1^-},d_x\phi_j^k\big\rangle_j^p- \delta_{k,p}H_{j+\frac{1}{2}}^{n+1^-} + \delta_{k,0}H_{j-\frac{1}{2}}^{n+1^-}\Big] \leq 0,
\end{equation}

\noindent with $H_{j\pm\frac{1}{2}}^{n+1^-}$ given by (\ref{eq:entropy_flux_acq_step}b).
\end{lemma}

\begin{proof}
Using (\ref{eq:entropy_flux_acq_step}a) and (\ref{eq:discr_tot_nrj}), the specific total energy may be written as $E_j^{k,n}=e(\tau_j^{k,n})+(\eta_j^{k,n}-(\Pi_j^{k,n})^2)/2a^2$. Hence, we have
\begin{eqnarray*}
 E_j^{k,n+1^-} - E_j^{k,n} &=& e(\tau_j^{k,n+1^-}) - e(\tau_j^{k,n}) + \frac{\eta_j^{k,n+1^-}-\eta_j^{k,n}}{2a^2} \\ 
 &&- \frac{(\Pi_j^{k,n+1^-}-\Pi_j^{k,n})^2}{2a^2} - \frac{\Pi_j^{k,n}(\Pi_j^{k,n+1^-}-\Pi_j^{k,n})}{a^2}.
\end{eqnarray*}

Using (\ref{eq:riemann_inv_eqn_acou}b) to substitute $\Pi_j^{k,n+1^-}-\Pi_j^{k,n}=-a^2(\tau_j^{k,n+1^-}-\tau_j^{k,n})$ and the fact that data at equilibrium impose $\Pi_j^{k,n}=\mathrm{p}(\tau_j^{k,n})=-e'(\tau_j^{k,n})$, we have
\begin{eqnarray*}
 E_j^{k,n+1^-} - E_j^{k,n} &-& \frac{\eta_j^{k,n+1^-}-\eta_j^{k,n}}{2a^2} =  e(\tau_j^{k,n+1^-}) - e(\tau_j^{k,n})\\  &-& e'(\tau_j^{k,n})(\tau_j^{k,n+1^-}-\tau_j^{k,n})-\frac{a^2}{2}(\tau_j^{k,n+1^-}-\tau_j^{k,n})^2.
\end{eqnarray*}

Applying a second-order Taylor development with integral remainder of $e(\tau_j^{k,n+1^-})$ about $\tau_j^{k,n}$, we obtain
\begin{eqnarray*}
 E_j^{k,n+1^-} - E_j^{k,n} - \frac{\eta_j^{k,n+1^-}-\eta_j^{k,n}}{2a^2} = (\tau_j^{k,n+1^-}-\tau_j^{k,n})^2 \times \cdots \\ \int_0^1\big(e''(\tau_j^{k,n}+\xi(\tau_j^{k,n+1^-}-\tau_j^{k,n}))-a^2\big)(1-\xi)d\xi \leq 0,
\end{eqnarray*}

\noindent under the subcharacteristic condition (\ref{eq:sub_cond_discr}). Finally, using (\ref{eq:entropy_ineq_acq_step}) gives (\ref{eq:energy_ineq_acq_step}). \qed
\end{proof}

\begin{theorem}
\label{th:pos_ineq_LPDG}
Assume that $\rho_{j\in\mathbb{Z}}^{0\leq k\leq p,n}>0$, then under the CFL condition (\ref{eq:CFL_LPDG}) and subcharacteristic condition (\ref{eq:sub_cond_discr}), the LPDG scheme satisfies positivity for the mean value of the solution:
\begin{equation}\label{eq:positivity_LPDG}
 \ol{\rho}_j^{n+1} > 0,
\end{equation}
\noindent and the discrete entropy inequality
\begin{equation}\label{eq:energy_ineq_LPDG}
 \rho E(\ol{\bf u}_j^{n+1}) - \ol{\rho E}_j^n + \lambda\Big(u_{j+\frac{1}{2}}^\star(\widehat{\rho E}_{j+\frac{1}{2}}^{n+1^-}+\Pi_{j+\frac{1}{2}}^\star) - u_{j-\frac{1}{2}}^\star(\widehat{\rho E}_{j-\frac{1}{2}}^{n+1^-}+\Pi_{j-\frac{1}{2}}^\star)\Big) \leq 0,
\end{equation}
\noindent where $\widehat{\rho E}_{j\pm\frac{1}{2}}^{n+1^-}$ denote upwind fluxes of the form (\ref{eq:upwind_flux}) evaluated at time $t^{(n+1^-)}$.
\end{theorem}

\begin{proof}
From assumptions of Theorem~\ref{th:pos_ineq_LPDG}, the results of Lemmas~\ref{th:pos_ineq_acq} and \ref{th:entropy_ineq_ac_step} hold. We thus infer positivity in (\ref{eq:positivity_LPDG}) by using the convex combination (\ref{eq:convex_comb_trp}) with $\rho_{j\in\mathbb{Z}}^{0\leq k\leq p,n+1^-}>0$.

Then, the entropy inequality (\ref{eq:energy_ineq_LPDG}) follows from the following arguments. Multiplying (\ref{eq:energy_ineq_acq_step}) with $\rho_j^{k,n}=L_j^{k,n+1^-}\rho_j^{k,n+1^-}$ and using (\ref{eq:def_Ljkn}a), we get
\begin{eqnarray*}
 \rho E_j^{k,n+1^-} - \rho E_j^{k,n} + \lambda_k\rho E_j^{k,n+1^-}\Big[\big\langle\partial_xu_h^{n+1^-},\phi_j^k\big\rangle_j^p && \\ +\delta_{k,p}(u_{j+\frac{1}{2}}^\star-U_j^{p,n+1^-}) -\delta_{k,0}(u_{j-\frac{1}{2}}^\star-U_j^{0,n+1^-})\Big] && \\ + \lambda_k\Big[-\big\langle\Pi_h^{n+1^-}u_h^{n+1^-},d_x\phi_j^k\rangle_j^p + \delta_{k,p}u_{j+\frac{1}{2}}^\star\Pi_{j+\frac{1}{2}}^\star - \delta_{k,0}u_{j-\frac{1}{2}}^\star\Pi_{j-\frac{1}{2}}^\star\Big] &\leq& 0.
\end{eqnarray*}

The first volume integral may be rewritten as
\begin{eqnarray*}
 \rho E_j^{k,n+1^-}\big\langle\partial_xu_h^{n+1^-},\phi_j^k\big\rangle_j^p &=& \frac{\omega_kh}{2}\rho E_j^{k,n+1^-}\frac{\partial u_h}{\partial x}\Big|_{x_j^k}^{n+1^-} = \big\langle\rho E_h^{n+1^-}\partial_xu_h^{n+1^-},\phi_j^k\big\rangle_j^p,
\end{eqnarray*}

\noindent and using integration by parts one thus obtains
\begin{eqnarray*}
 \rho E_j^{k,n+1^-} - \rho E_j^{k,n} + \lambda_k\Big[-\big\langle u_h^{n+1^-},\partial_x(\rho E_h^{n+1^-}\phi_j^k)+\Pi_h^{n+1^-}d_x\phi_j^k\big\rangle_j^p && \\ 
 + \delta_{k,p}u_{j+\frac{1}{2}}^\star(\rho E_j^{p,n+1^-}+\Pi_{j+\frac{1}{2}}^\star) - \delta_{k,0}u_{j-\frac{1}{2}}^\star(\rho E_j^{0,n+1^-}+\Pi_{j-\frac{1}{2}}^\star)\Big] &\leq& 0.
\end{eqnarray*}

Summing over $0\leq k\leq p$ with weights $\omega_k/2$, one obtains
\begin{eqnarray}\label{eq:energy_ineq_acq_step2}
 \ol{\rho E}_j^{n+1^-} - \ol{\rho E}_j^{n} + \lambda\Big(-\big\langle u_h^{n+1^-},\partial_x\rho E_h^{n+1^-}\big\rangle_j^p  &&\nonumber\\ + u_{j+\frac{1}{2}}^\star(\rho E_j^{p,n+1^-}+\Pi_{j+\frac{1}{2}}^\star)
 - u_{j-\frac{1}{2}}^\star(\rho E_j^{0,n+1^-}+\Pi_{j-\frac{1}{2}}^\star)\Big) &\leq& 0.
\end{eqnarray}

Now, by convexity of the mapping $\rho E({\bf u})$, the convex combination (\ref{eq:convex_comb_trp}) gives
\begin{eqnarray*}
 {\rho E}(\ol{\bf u}_j^{n+1}) - \sum_{k=0}^p\bigg(\frac{\omega_k}{2}-\lambda\Big(\big\langle u_h^{n+1^-},d_x\phi_j^k\big\rangle_j^p-\delta_{k,p}(u_{j+\frac{1}{2}}^\star)^- && \\ +\delta_{k,0}(u_{j-\frac{1}{2}}^\star)^+\Big)\bigg)\rho E_j^{k,n+1^-} +\lambda (u_{j+\frac{1}{2}}^\star)^-\rho E_{j+1}^{0,n+1^-} - \lambda (u_{j-\frac{1}{2}}^\star)^+\rho E_{j-1}^{p,n+1^-} \leq 0.
\end{eqnarray*}

Using integration by parts, we obtain
\begin{eqnarray*}
 {\rho E}(\ol{\bf u}_j^{n+1}) &-& \ol{\rho E}_j^{n+1^-} + \lambda\sum_{k=0}^p\Big(-\big\langle\partial_x u_h^{n+1^-},\phi_j^k\big\rangle_j^p \\ &-&\delta_{k,p}((u_{j+\frac{1}{2}}^\star)^--U_j^{p,n+1^-}) +\delta_{k,0}((u_{j-\frac{1}{2}}^\star)^+-U_j^{0,n+1^-})\Big)\rho E_j^{k,n+1^-} \\ &+& \lambda (u_{j+\frac{1}{2}}^\star)^-\rho E_{j+1}^{0,n+1^-} - \lambda (u_{j-\frac{1}{2}}^\star)^+\rho E_{j-1}^{p,n+1^-} \leq 0.
\end{eqnarray*}

Using again integration by parts and the definition of the numerical flux (\ref{eq:upwind_flux}) applied to $\rho E_h$, one obtains
\begin{eqnarray}\label{eq:energy_ineq_trp_step}
 {\rho E}(\ol{\bf u}_j^{n+1}) - \ol{\rho E}_j^{n+1^-} + \lambda \Big(\big\langle u_h^{n+1^-},\partial_x\rho E_h^{n+1^-}\big\rangle_j^p && \nonumber\\ +(u_{j+\frac{1}{2}}^\star)^-(\widehat{\rho E}_{j+\frac{1}{2}}^{n+1^-}\!-\!\rho E_{j}^{p,n+1^-}) 
-(u_{j-\frac{1}{2}}^\star)^+(\widehat{\rho E}_{j-\frac{1}{2}}^{n+1^-}\!-\!\rho E_{j}^{0,n+1^-}) \Big) &\leq& 0.
\end{eqnarray}

Summing (\ref{eq:energy_ineq_acq_step2}) and (\ref{eq:energy_ineq_trp_step}) gives (\ref{eq:energy_ineq_LPDG}). \qed
\end{proof}


%
\subsection{High-order time discretization}\label{sec:HO_time_discr}

The present method is extended to high-order time integration by using strong-stability preserving explicit Runge-Kutta methods \cite{shu-osher88,spiteri_ruuth02}. These methods consist in convex combinations of first-order forward Euler methods and thus will keep positivity of Theorem~\ref{th:pos_ineq_LPDG} under a given CFL condition. We note however that the first-order time discretization (\ref{eq:discrete_eq_Euler}) is not an explicit forward Euler method because the residuals are evaluated at an intermediate time step $t^{(n+1^-)}$. As a consequence, accuracy in time is not guaranteed when using high-order Runge-Kutta schemes. The design of adapted high-order time integration is beyond the scope of the present study. However, the numerical experiments in section~\ref{sec:num_xp} tend to indicate that explicit Runge-Kutta time integration do not alter accuracy of the present method and are thus well adapted in practice at least for the present range of applications.

\subsection{Limiting strategy}\label{sec:limiters}

The properties in Theorem~\ref{th:pos_ineq_LPDG} hold only for the mean value in mesh elements of the numerical solution at time $t^{(n+1)}$, which is not sufficient for robustness and stability of numerical computations. However, these results may motivate the use of \textit{a posteriori} limiters introduced in \cite{zhang_shu_10a,zhang_shu_10b}. These limiters aim at extending preservation of invariant domains \cite{zhang_shu_10b} or maximum-principle  \cite{zhang_shu_10a} from mean to nodal values within elements. Our strategy differs slightly from the ones in \cite{zhang_shu_10a,zhang_shu_10b}, so we describe it in the following.

First, we enforce positivity of nodal values of density by using the linear limiter 
\begin{equation}\label{eq:pos_limiter}
 \breve{\bf U}_j^{k,n+1} = \theta_j^\rho({\bf U}_j^{k,n+1}- \ol{\bf u}_j^{n+1}) + \ol{\bf u}_j^{n+1},
\end{equation}

\noindent with $0\leq \theta_j^\rho\leq1$ defined by
\begin{equation*}
 \theta_j^\rho = \min\Big(\frac{\ol\rho_j^n-\epsilon}{\ol\rho_j^n-\rho_j^{min}},1\Big), \quad \rho_j^{min}=\min_{0\leq k\leq p} \rho_j^{k,n+1},
\end{equation*}

\noindent and $0<\epsilon\ll1$ a parameter.

Then, we strengthen the entropy inequality (\ref{eq:energy_ineq_LPDG}) by observing that the discrete transport step (\ref{eq:discrete_eq_trp}) satisfies a maximum principle for any convex function ${\cal U}:\Omega^a\ni{\bf u}\rightarrow {\cal U}({\bf u})\in\mathbb{R}$. Indeed, using (\ref{eq:convex_comb_trp}) we obtain
\begin{equation*}
 {\cal U}(\ol{\bf u}_j^{n+1}) \leq {\cal U}_j^{n+1^-}:= \max\big({\cal U}({\bf U}_{j-1}^{p,n+1^-}),{\cal U}({\bf U}_j^{0\leq k\leq p,n+1^-}),{\cal U}({\bf U}_{j+1}^{0,n+1^-})\big),
\end{equation*}

\noindent we thus impose a maximum principle at nodal values from
\begin{equation}\label{eq:entropy_limiter}
 \tilde{\bf U}_j^{k,n+1} = \theta_j^s(\breve{\bf U}_j^{k,n+1}- \ol{\bf u}_j^{n+1}) + \ol{\bf u}_j^{n+1},
\end{equation}
\noindent with $0\leq \theta_j^s\leq1$ defined by
\begin{equation*}
 \theta_j^s = \min_ {0\leq k\leq p}\Big(\theta_{j}^{s,k}: \quad {\cal U}\big(\theta_{j}^{s,k}(\breve{\bf U}_j^{k,n+1}- \ol{\bf u}_j^{n+1}) + \ol{\bf u}_j^{n+1}\big) =  {\cal U}_j^{n+1^-}\Big).
\end{equation*}

Note that when $\breve{\bf U}_j^{k,n+1}$ is not in $\Omega^a$, there exists a unique $0\leq\theta_{j}^{s,k}\leq1$ such that the above relation holds by convexity of ${\cal U}$ since $\ol{\bf u}_j^{n+1}$ is in $\Omega^a$. In practice, we use the total energy as entropy. Finally, we replace the DOFs at time $t^{(n+1)}$ by the limited values $ \tilde{\bf U}_{j\in\mathbb{Z}}^{0\leq k\leq p,n+1}$. For high-order time integration, we apply the limiter after each stage of the Runge-Kutta scheme. We stress that the limiters (\ref{eq:pos_limiter}) and (\ref{eq:entropy_limiter}) preserve conservation and accuracy for smooth solutions \cite{zhang_shu_10a,zhang_shu_10b}.

We end this section by summing up our strategy at the discrete level with the following algorithm applied at each stage of the Runge-Kutta method: 
\begin{enumerate}[(i)]
 \item solve the linear system (\ref{eq:riemann_inv_eqn_acou}) for the characteristic variables (\ref{eq:def_vec_cev_j}) with data at equilibrium (\ref{eq:projections_n_n+1^-}a); 
 \item compute the conservative variables with (\ref{eq:def_vec_cev_j}); 
 \item compute the discrete residuals of the LPDG scheme (\ref{eq:discrete_eq_Euler}) with these values; 
 \item apply the limiters (\ref{eq:pos_limiter}) and (\ref{eq:entropy_limiter}).
\end{enumerate}

%
%
\section{Numerical experiments}\label{sec:num_xp}

In this section we present several numerical experiments to illustrate the performances of the LPDG scheme derived in this work. For all experiments, we consider an isentropic polytropic ideal gas with an equation of state of the form $\mathrm{p}(\tau)=\kappa\tau^{-\gamma}$ with $\kappa>0$ and $\gamma>1$.

We use strong-stability preserving Runge-Kutta time integration schemes of order $p+1$ when using polynomials of degree $p$ for the space discretization: the two-stage second-order Heun method for $p=1$, the three-stage third-order scheme of Shu-Osher \cite{shu-osher88} for $p=2$, and the five-stage fourth-order scheme of Spiteri and Ruuth \cite{spiteri_ruuth02} for $p=3$, respectively.

Finally, the \textit{a priori} CFL condition (\ref{eq:CFL_LPDG}) and subcharacteristic condition (\ref{eq:sub_cond_discr}) are imposed at time $t^{(n)}$ as was proposed in \cite{coquel_etal_10,chalons_etal_13}.



%
\subsection{Manufactured smooth solution}

We first consider the convection of a density wave in a uniform flow with Mach number $M_\infty$. Let $\Omega=(0,1)$, we solve
\begin{equation*}
 \partial_t {\bf u} + \partial_x{\bf f}({\bf u}) = {\bf s}, \quad \mbox{in }\Omega\times(0,\infty),
\end{equation*}

\noindent with periodicity conditions and initial condition
\begin{equation*}
 \rho_0(x) = 1+\epsilon\sin(2\pi x), \quad u_0(x) = 1, \quad \forall x\in\Omega,
\end{equation*}

\noindent with $\epsilon=0.2$. The source term ${\bf s}={\bf s}(x,t)$ is such that the exact solution for this problem reads 
\begin{equation*}
 \rho(x,t) = 1+\epsilon\sin(2\pi(x-t)), \quad u(x,t) = 1, \quad \forall x\in\Omega, t>0.
\end{equation*}

The parameters of the equation of state are $\kappa=1/\gamma M_\infty^2$ with $M_\infty=0.1$ and $\gamma=1.4$. Table~\ref{tab:density_wave_error} indicates different norms of the numerical error on density $e_h=\rho_h-\rho$ for different polynomial degrees and grid refinements with associated convergence orders in space. The expected $p+1$ order of convergence is recovered with the present method. 

\begin{table}
     \begin{center}
     \caption{Manufactured smooth solution: different norms of the error at time $t=5$ and associated orders of convergence.}
     \begin{tabular}{clcccccc}
        \noalign{\smallskip}\hline\noalign{\smallskip}
 	$p$ & $h$ & $\|e_h\|_{L^1(\Omega)}$  & ${\cal O}_1$ & $\|e_h\|_{L^2(\Omega)}$  & ${\cal O}_2$ & $\|e_h\|_{L_\infty(\Omega)}$  & ${\cal O}_\infty$ \\
        \noalign{\smallskip}\hline\noalign{\smallskip}
  	    & $1/4$	& $0.35511e\!-\!01$ & $-$    & $0.45109e\!-\!01$ & $-$    & $0.85997e\!-\!01$ & $-$   \\
  	    & $1/8$	& $0.97747e\!-\!02$ & $1.86$ & $0.12770e\!-\!01$ & $1.82$ & $0.31154e\!-\!01$ & $1.46$\\
  	  1 & $1/16$	& $0.24729e\!-\!02$ & $1.98$ & $0.32942e\!-\!02$ & $1.95$ & $0.78816e\!-\!02$ & $1.98$\\
  	    & $1/32$	& $0.61003e\!-\!03$ & $2.01$ & $0.81665e\!-\!03$ & $2.01$ & $0.19587e\!-\!02$ & $2.00$\\
  	    & $1/64$	& $0.14369e\!-\!03$ & $2.08$ & $0.19544e\!-\!03$ & $2.06$ & $0.45383e\!-\!03$ & $2.10$\\
        \noalign{\smallskip}\hline\noalign{\smallskip}
  	    & $1/4$	& $0.50087e\!-\!02$ & $-$    & $0.76828e\!-\!02$ & $-$    & $0.26541e\!-\!01$ & $-$   \\ 
  	    & $1/8$	& $0.42909e\!-\!03$ & $3.54$ & $0.65902e\!-\!03$ & $3.54$ & $0.24946e\!-\!02$ & $3.41$\\
  	  2 & $1/16$	& $0.33870e\!-\!04$ & $3.66$ & $0.53700e\!-\!04$ & $3.61$ & $0.21400e\!-\!03$ & $3.54$\\
  	    & $1/32$	& $0.50879e\!-\!05$ & $2.73$ & $0.79338e\!-\!05$ & $2.75$ & $0.25079e\!-\!04$ & $3.09$\\
  	    & $1/64$	& $0.50163e\!-\!06$ & $3.34$ & $0.75007e\!-\!06$ & $3.40$ & $0.23383e\!-\!05$ & $3.42$\\
        \noalign{\smallskip}\hline\noalign{\smallskip}
  	    & $1/4$	& $0.39559e\!-\!03$ & $-$    & $0.55401e\!-\!03$ & $-$    & $0.19667e\!-\!02$ & $-$   \\
  	    & $1/8$	& $0.23759e\!-\!04$ & $4.05$ & $0.36194e\!-\!04$ & $3.93$ & $0.13136e\!-\!03$ & $3.90$\\
  	  3 & $1/16$	& $0.15068e\!-\!05$ & $3.97$ & $0.22986e\!-\!05$ & $3.97$ & $0.94855e\!-\!05$ & $3.79$\\
  	    & $1/32$	& $0.94001e\!-\!07$ & $4.00$ & $0.14440e\!-\!06$ & $3.99$ & $0.61378e\!-\!06$ & $3.94$\\
  	    & $1/64$	& $0.58836e\!-\!08$ & $3.99$ & $0.90435e\!-\!08$ & $3.99$ & $0.38706e\!-\!07$ & $3.98$\\
        \noalign{\smallskip}\hline\noalign{\smallskip}
    \end{tabular}
    \label{tab:density_wave_error}
    \end{center}
\end{table}

\subsection{Riemann problems}

We now consider Riemann problems with initial condition
\begin{equation*}
 {\bf u}_0(x) = \left\{ \begin{array}{rl}  {\bf u}_L, & x<0, \\ {\bf u}_R, & x>0. \end{array} \right.
\end{equation*}

The set of initial conditions is given in Table~\ref{tab:RP_IC}. Problems RP1 and RP3 are taken from \cite{chalons_coulombel08}, while problem RP4 is taken from \cite{berthelin_etal15}. Figures~\ref{fig:RP1} to \ref{fig:RP5} compare the numerical solution in symbols with the exact solution in lines.

\begin{table}
     \begin{center}
     \caption{Initial conditions and physical parameters of Riemann problems}
     \begin{tabular}{llcccc}
        \noalign{\smallskip}\hline\noalign{\smallskip}
 	test & description & left state ${\bf u}_L$ & right state ${\bf u}_R$ & $\kappa$ & $\gamma$ \\
        \noalign{\smallskip}\hline\noalign{\smallskip}
	RP1 & shock-shock & $(1,1)^\top$ & $(2,0.5)^\top$ & $\tfrac{(\gamma-1)^2}{4\gamma}$ & 1.6\\
	RP2 & shock-shock & $(1,2)^\top$ & $(2,1)^\top$ & $\tfrac{(\gamma-1)^2}{4\gamma}$ & 1.6 \\
	RP3 & rarefaction-shock & $(1,-0.5)^\top$ & $(0.5,-0.5)^\top$ & $\tfrac{(\gamma-1)^2}{4\gamma}$ & 1.6 \\
	RP4 & rarefaction-rarefaction & $(1,-5)^\top$ & $(1,5)^\top$ & $1$ & 1.4 \\
        \noalign{\smallskip}\hline\noalign{\smallskip}
    \end{tabular}
     \label{tab:RP_IC}
    \end{center}
\end{table}

For RP1, RP2 and RP3, results are qualitatively similar. The shock waves are well captured and the increase in the discretization order has a clear positive effect on the approximation of the rarefaction waves in RP3. We observe some spurious oscillations of low amplitude in the neighborhood of strong shocks with the highest discretization order $p=3$ as visible in the density distributions of RP1 and RP2. The solution for RP4 is made of two symmetric rarefaction waves with formation of near-vacuum in the intermediate region. The positivity limiter is successful to keep robustness of the computation and increasing $p$ reduces the diffusion at the tail of the waves as expected. However, the entropy limiter alters the solution at the head of the rarefaction waves for $p=3$. We attribute this effect to the fact that the solution is not smooth in this region, the limiters keeping accuracy for smooth solutions only.

\begin{figure}
\begin{center}
\subfigure{\epsfig{figure=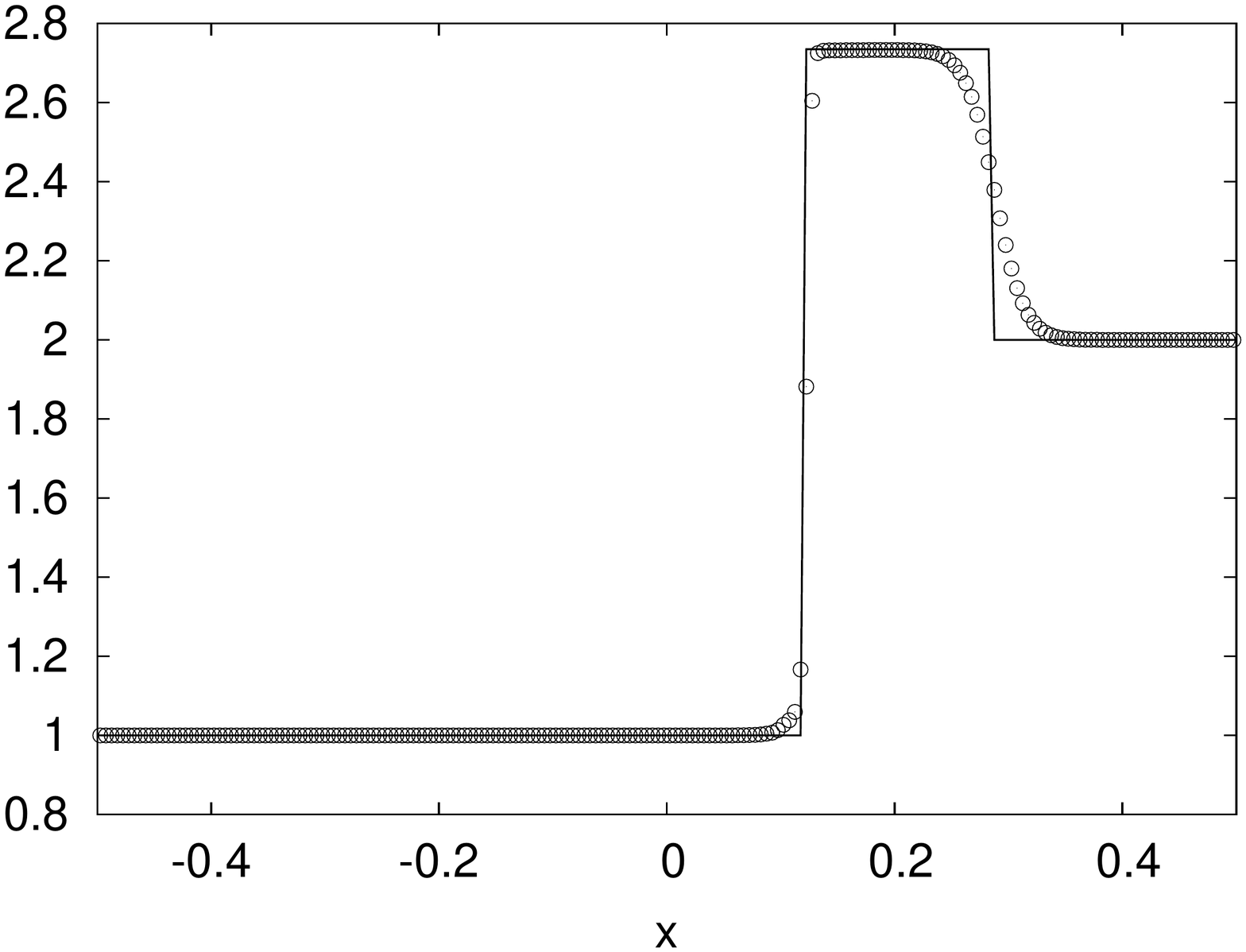,width=6.0cm}}
\subfigure{\epsfig{figure=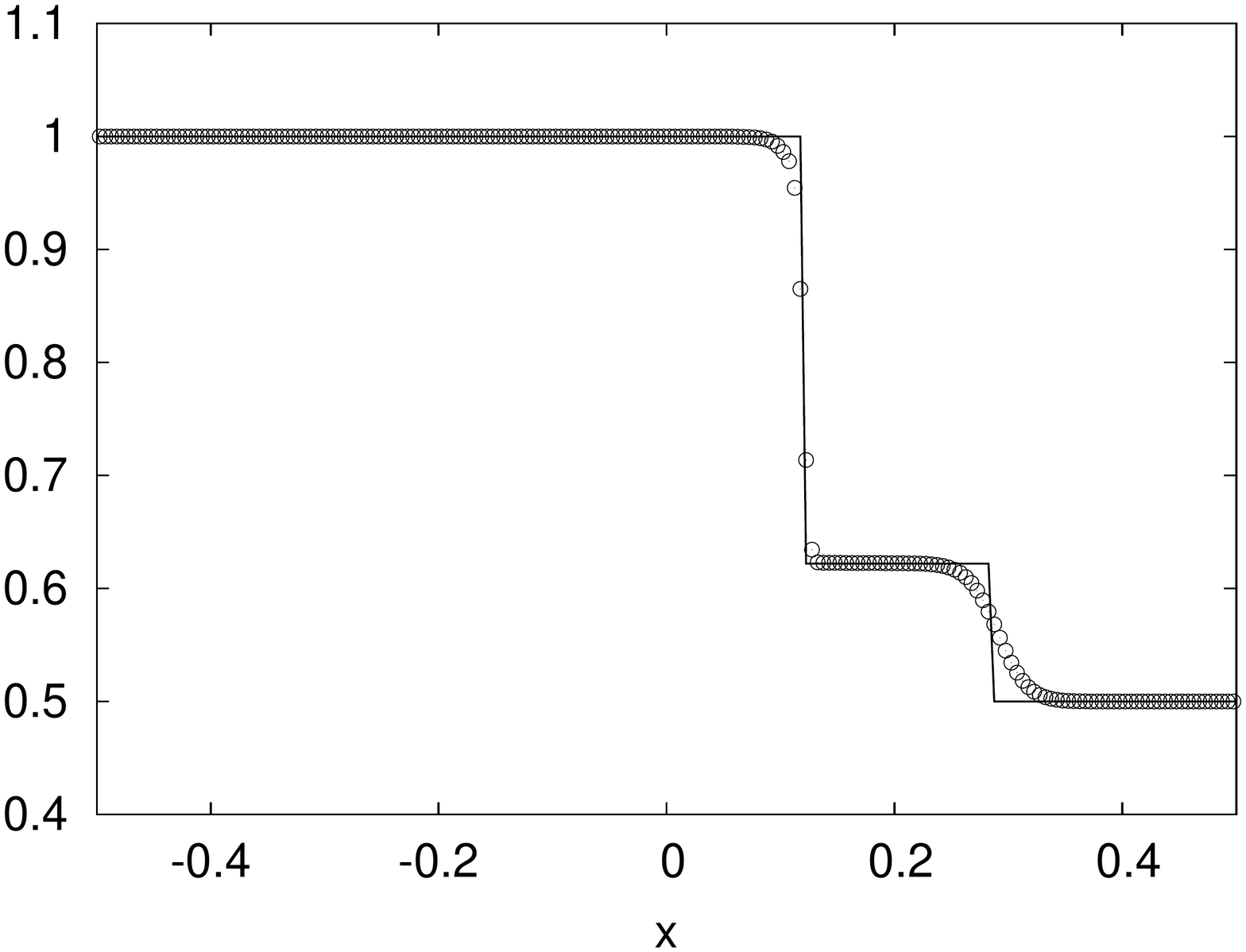  ,width=6.0cm}} \\ \vspace{-1cm}
\subfigure{\epsfig{figure=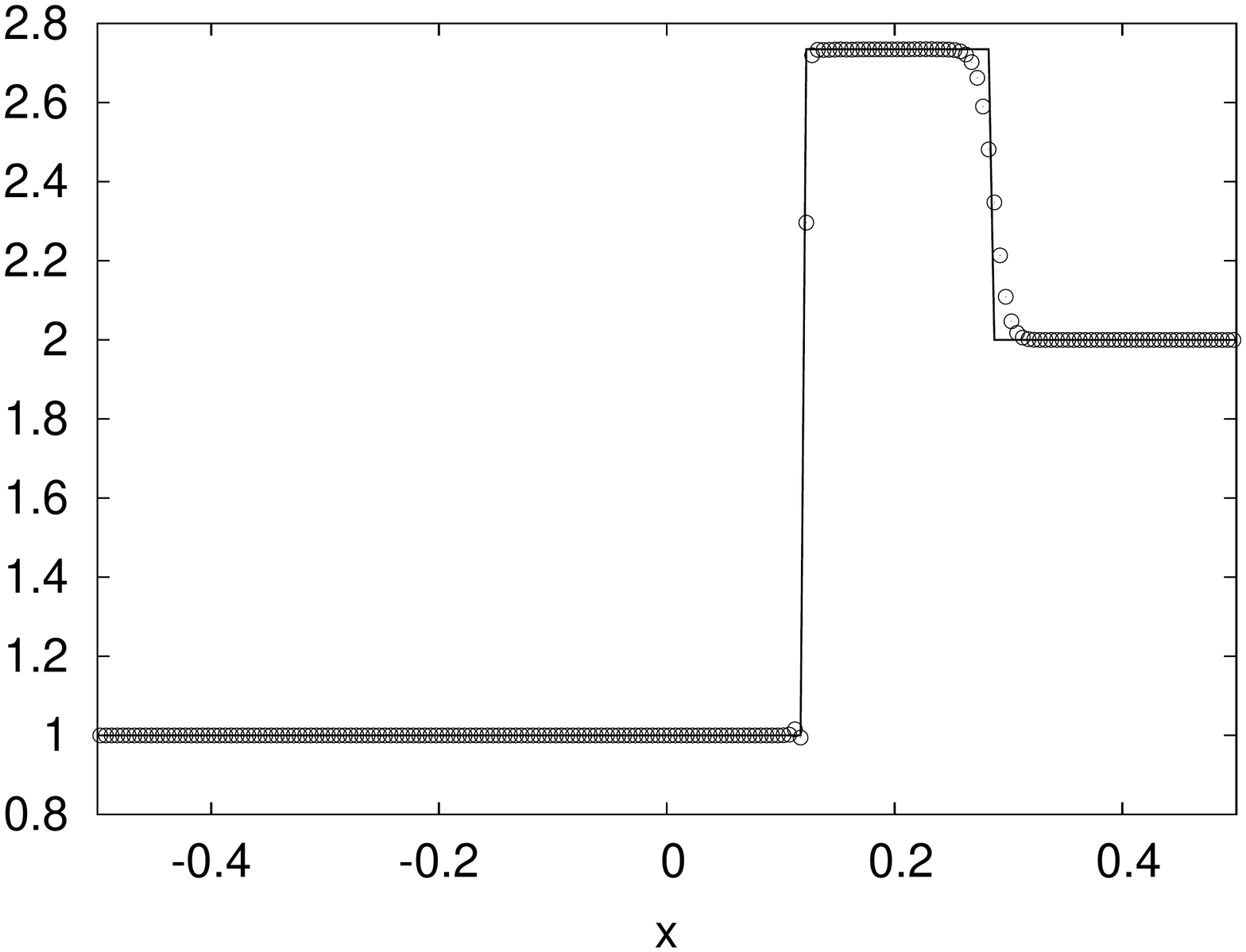,width=6.0cm}}
\subfigure{\epsfig{figure=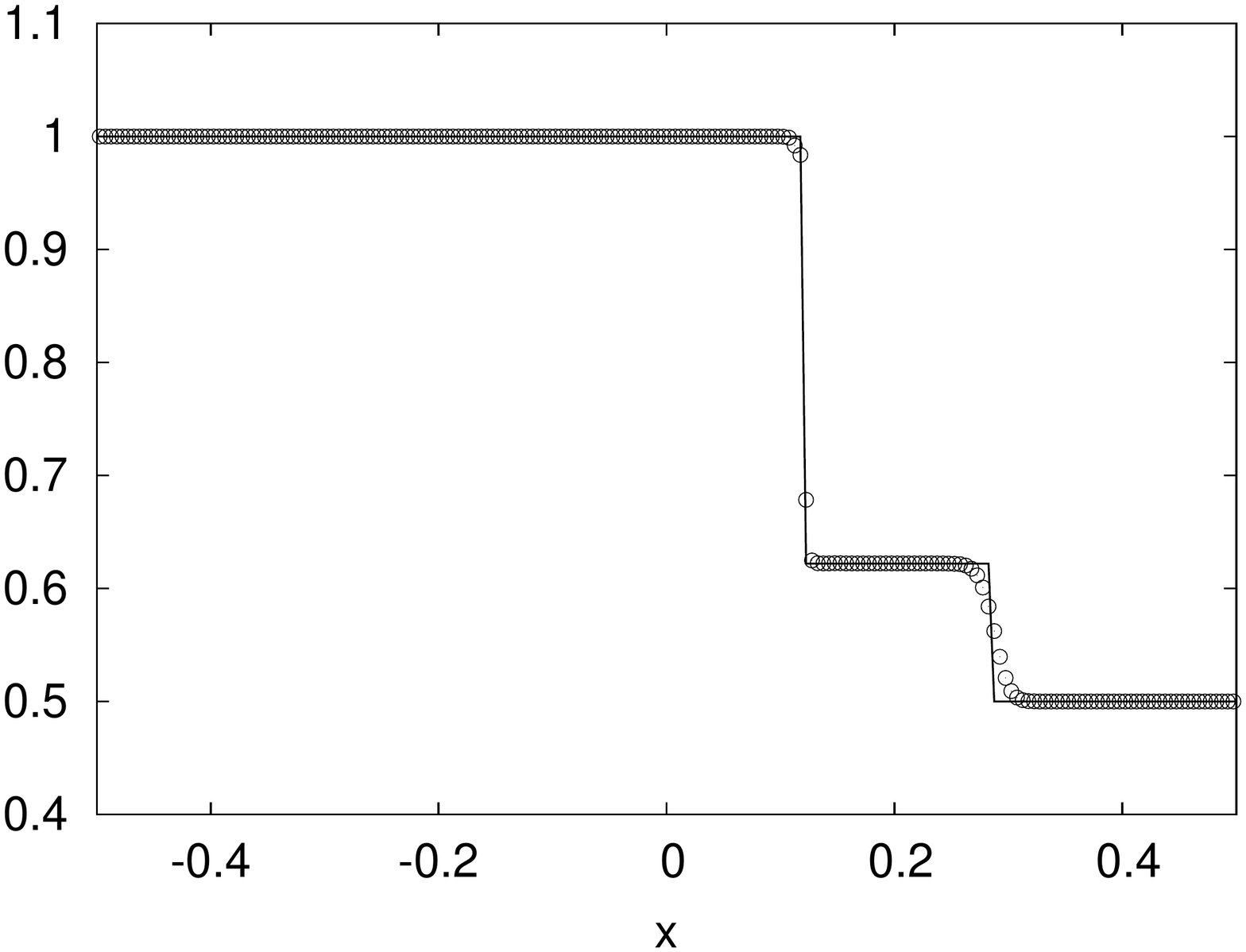  ,width=6.0cm}} \\ \vspace{-1cm}
\subfigure{\epsfig{figure=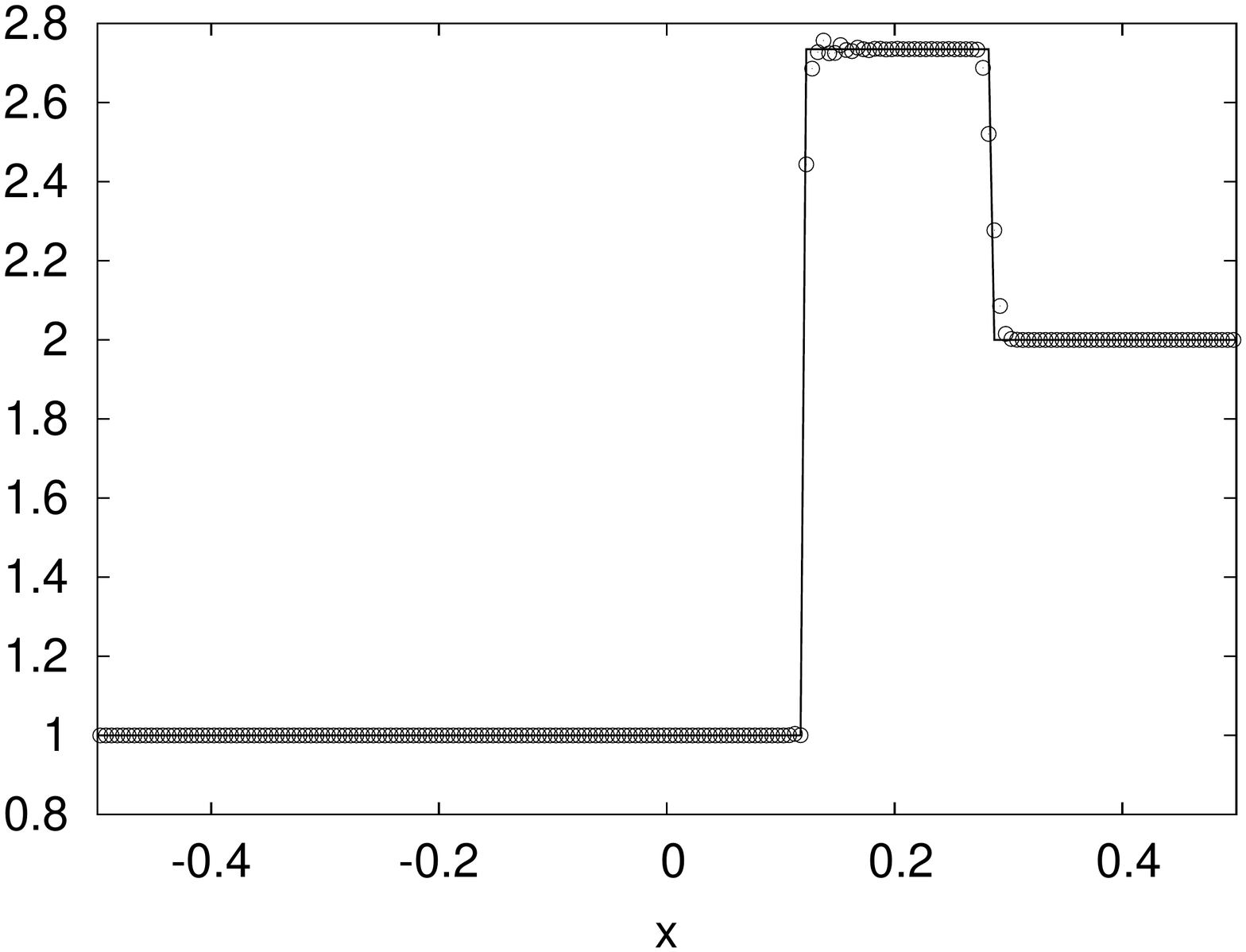,width=6.0cm}}
\subfigure{\epsfig{figure=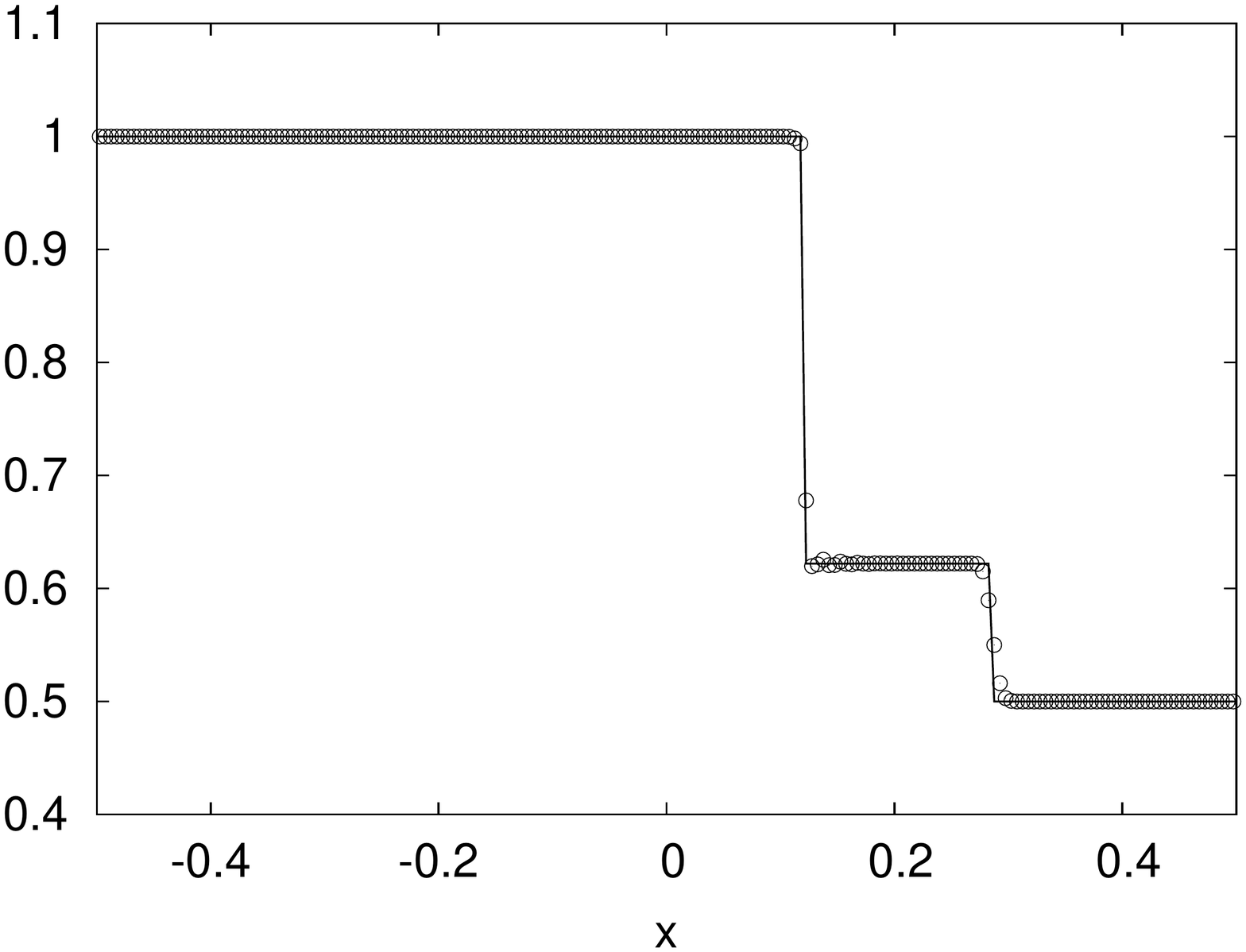  ,width=6.0cm}}
\caption{RP1: numerical solution for density (left) and velocity (right) at time $t=0.3$ for different polynomial degrees $p=1$ to $p=3$ from top to bottom, $h=\tfrac{1}{200}$.}
\label{fig:RP1}
\end{center}
\end{figure}

\begin{figure}
\begin{center}
\subfigure{\epsfig{figure=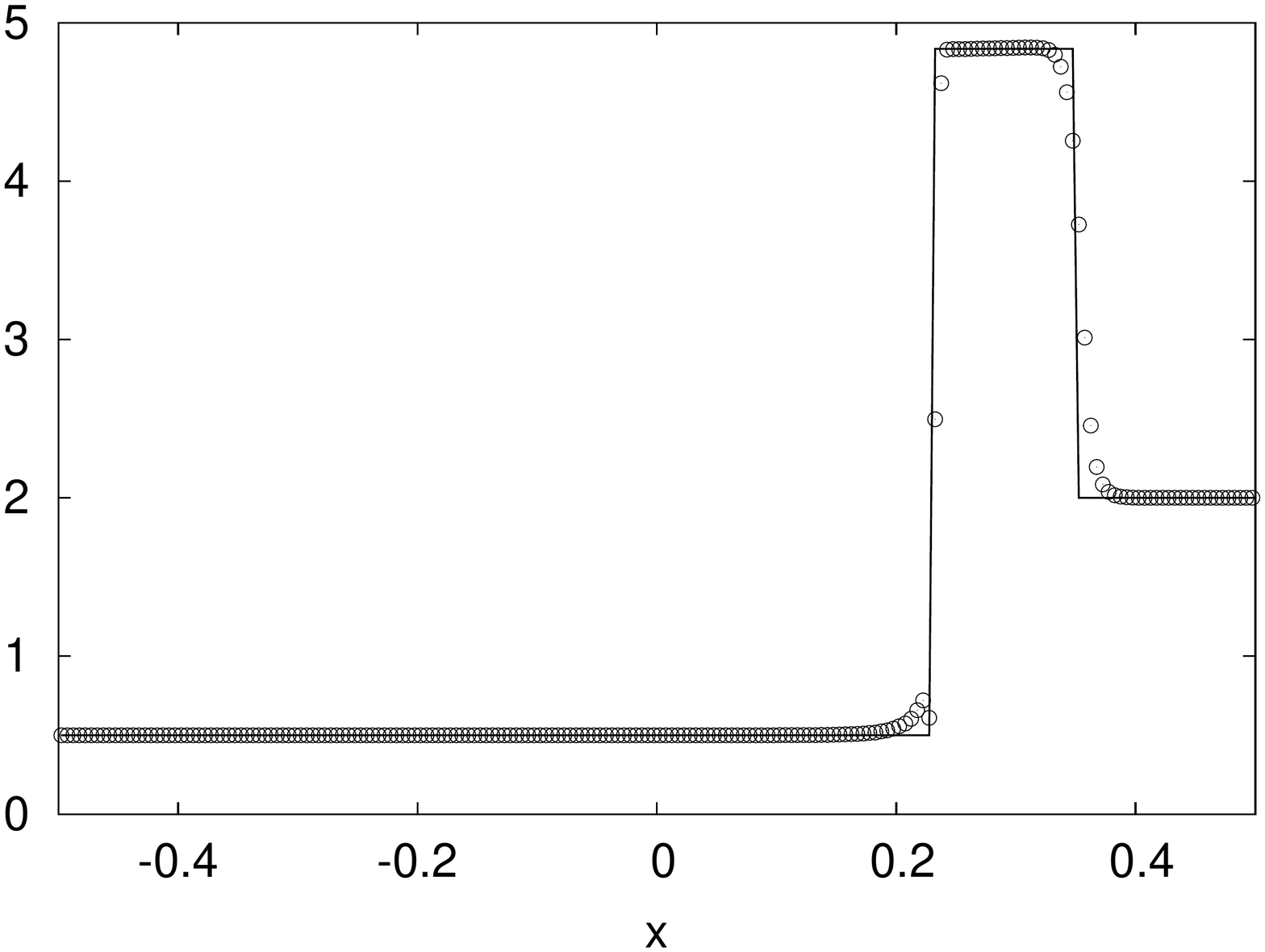,width=6cm}}
\subfigure{\epsfig{figure=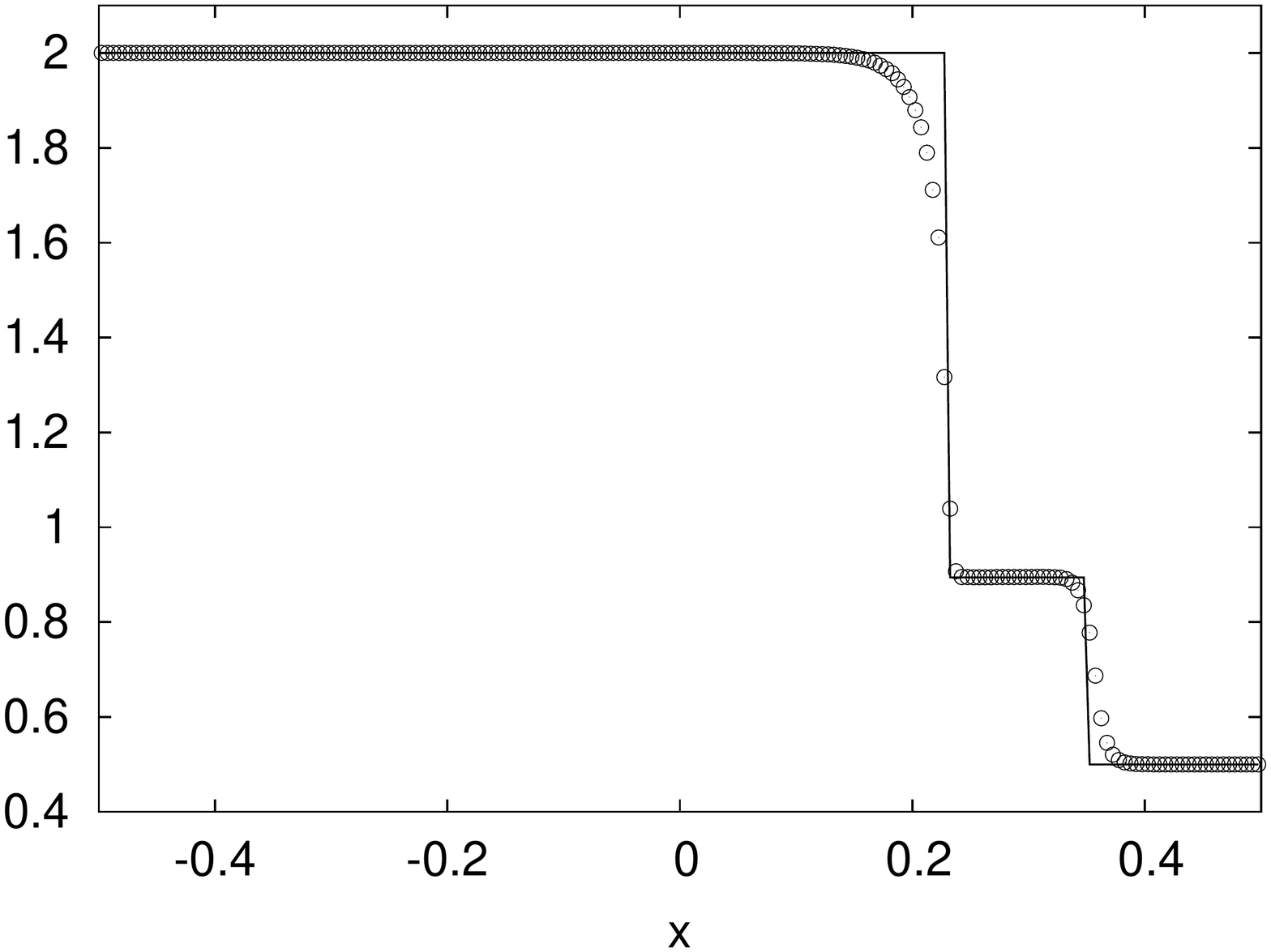  ,width=6cm}} \\ \vspace{-1cm}
\subfigure{\epsfig{figure=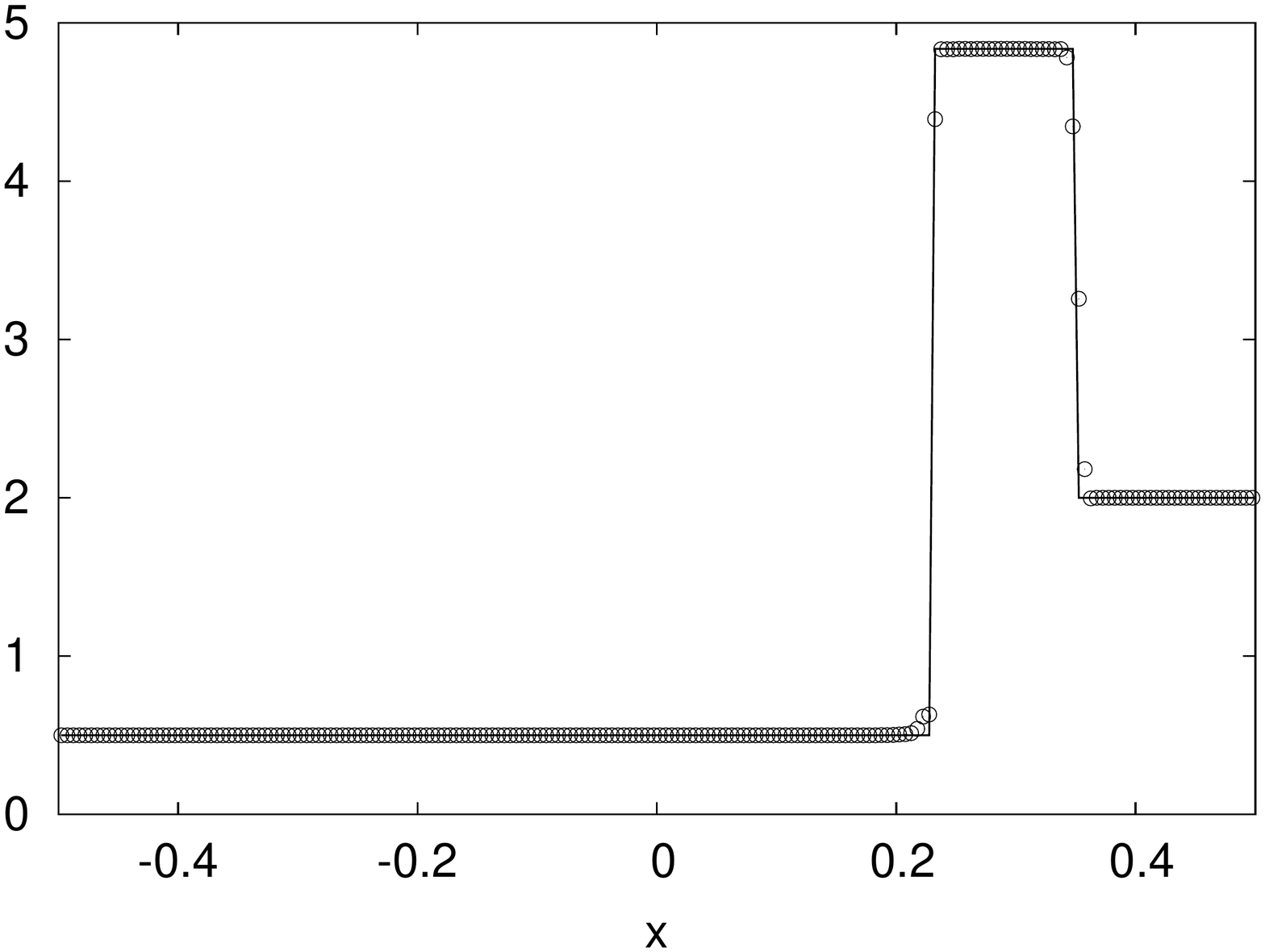,width=6cm}}
\subfigure{\epsfig{figure=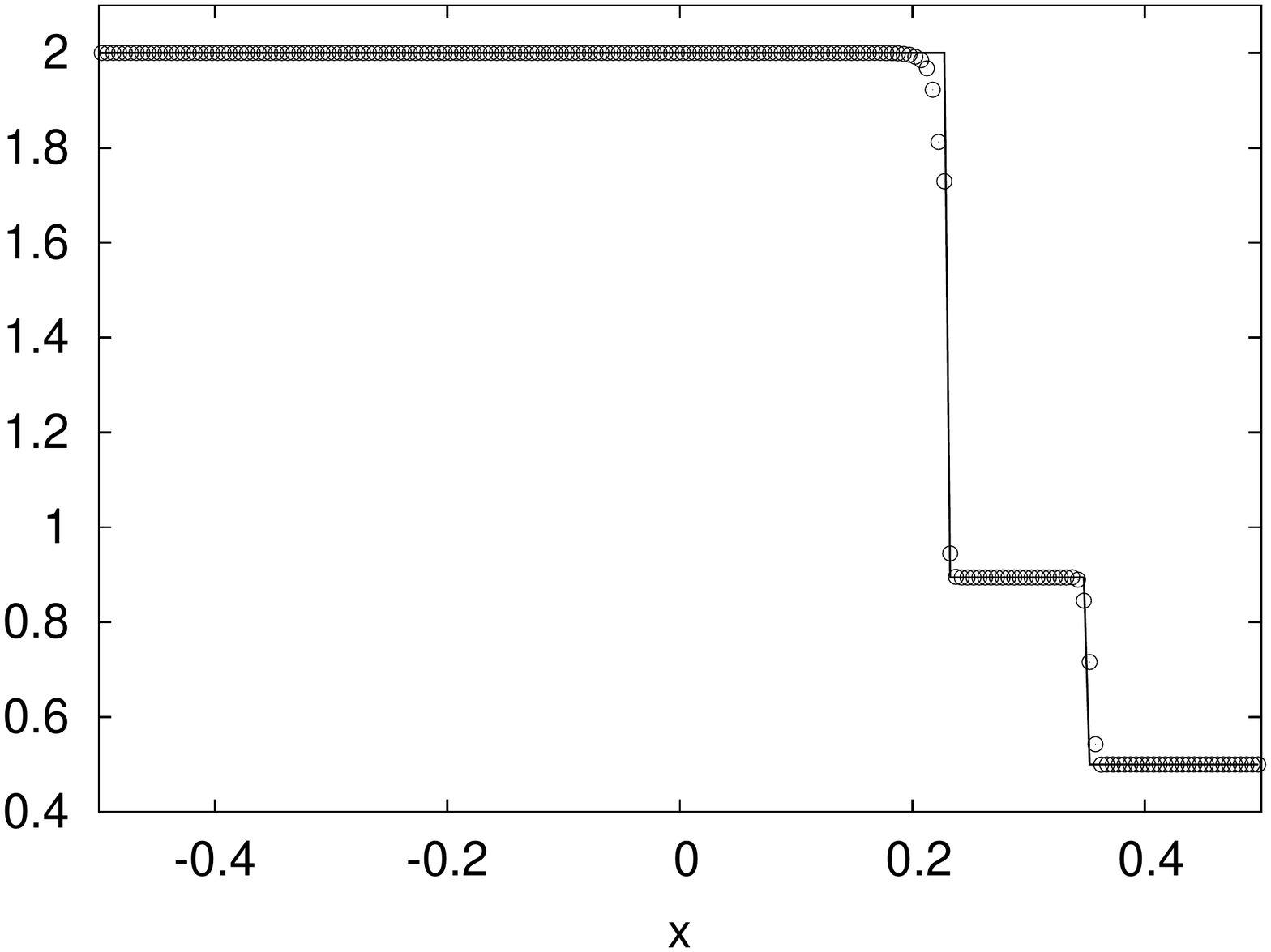  ,width=6cm}} \\ \vspace{-1cm}
\subfigure{\epsfig{figure=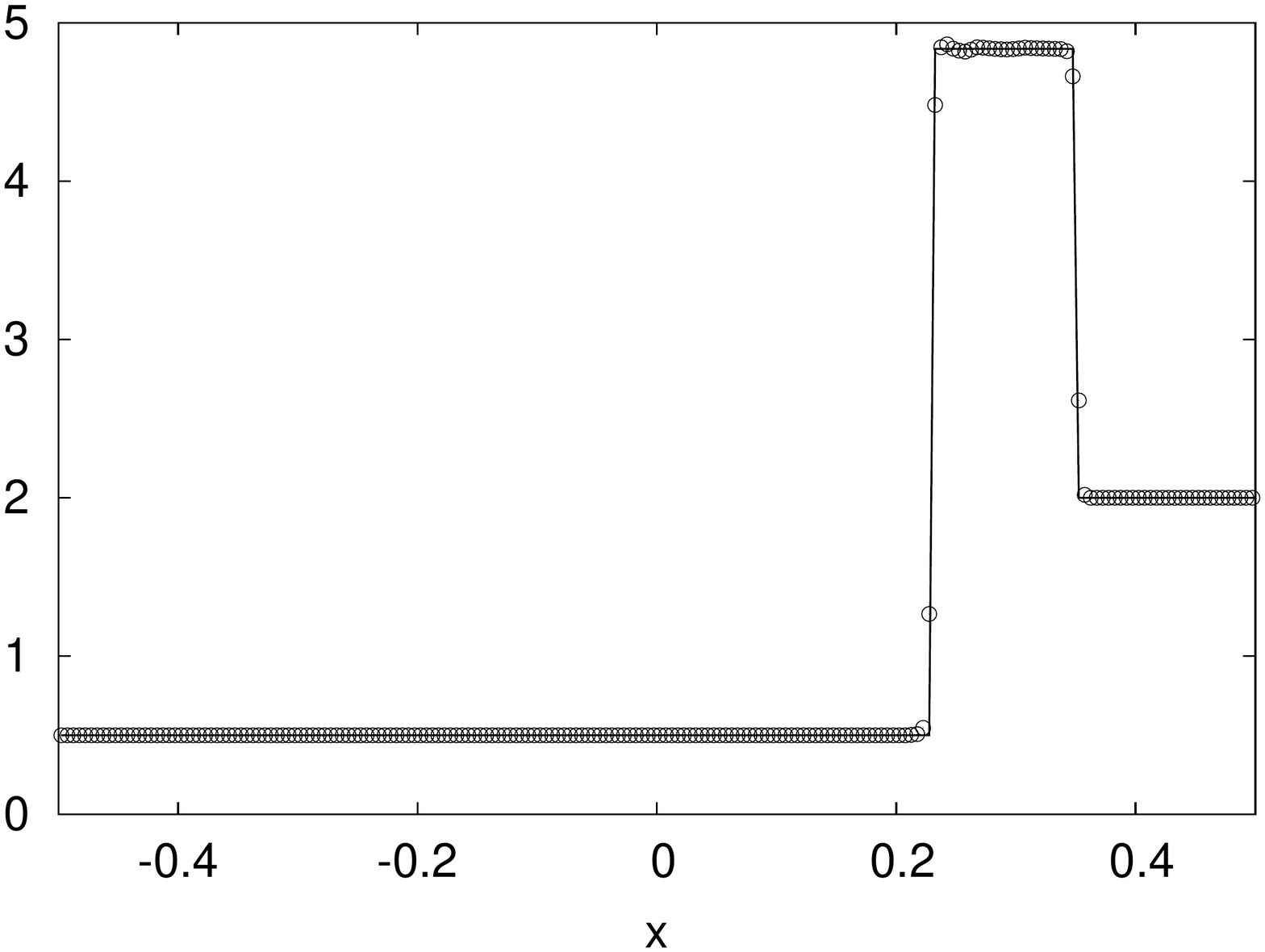,width=6cm}}
\subfigure{\epsfig{figure=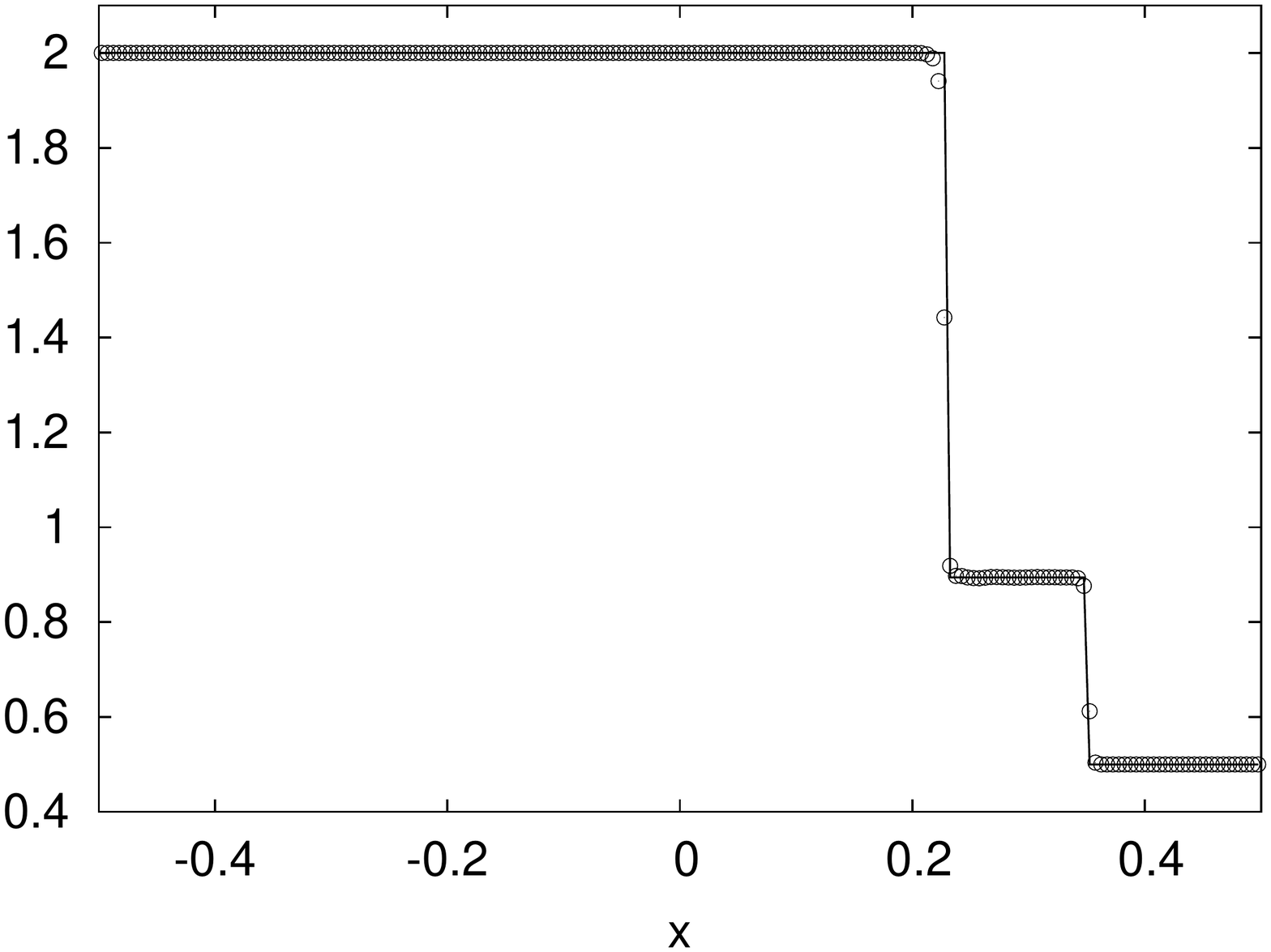  ,width=6cm}}
\caption{RP2: numerical solution for density (left) and velocity (right) at time $t=0.3$ for different polynomial degrees $p=1$ to $p=3$ from top to bottom, $h=\tfrac{1}{200}$.}
\label{fig:RP2}
\end{center}
\end{figure}

\begin{figure}
\begin{center}
\subfigure{\epsfig{figure=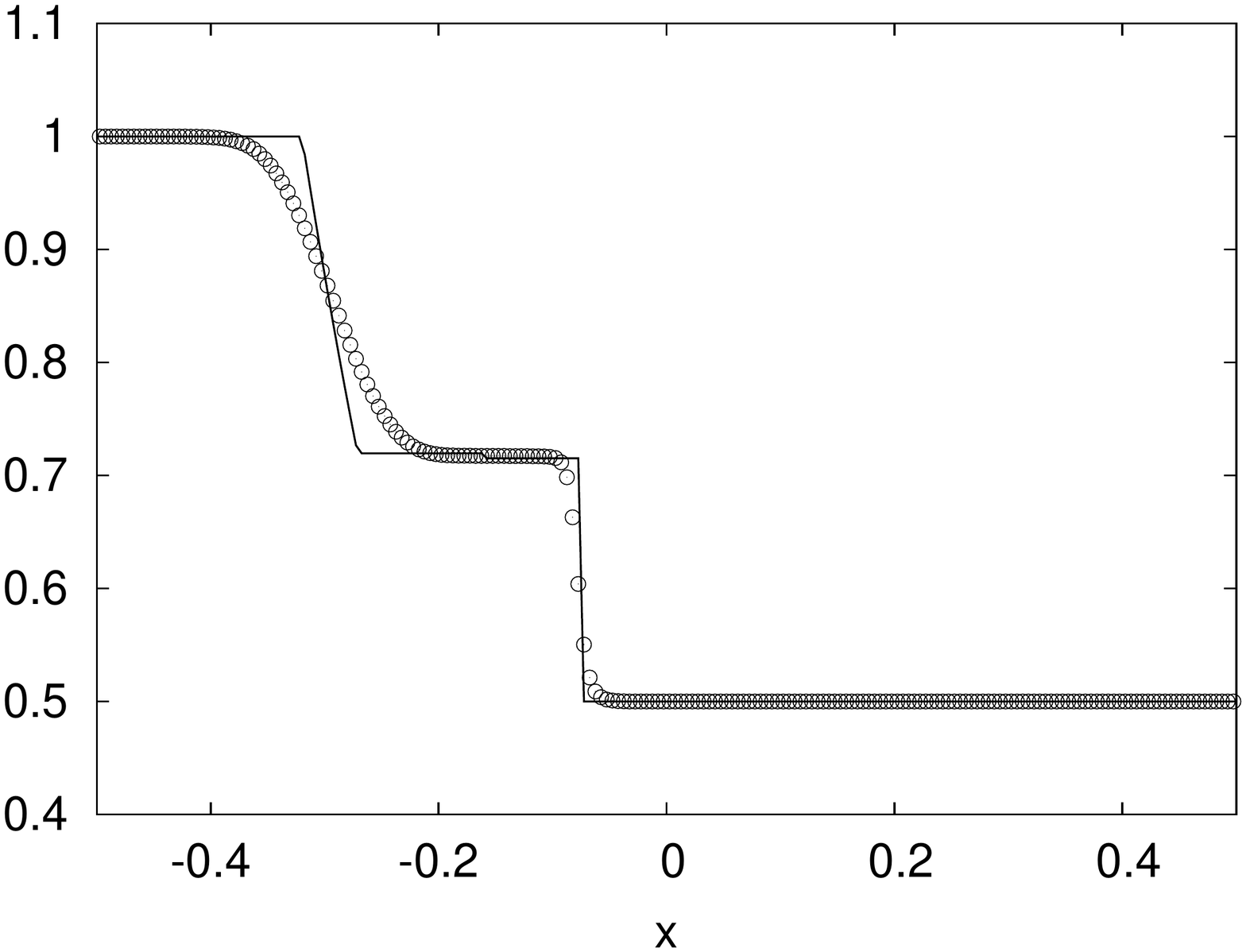,width=6cm}}
\subfigure{\epsfig{figure=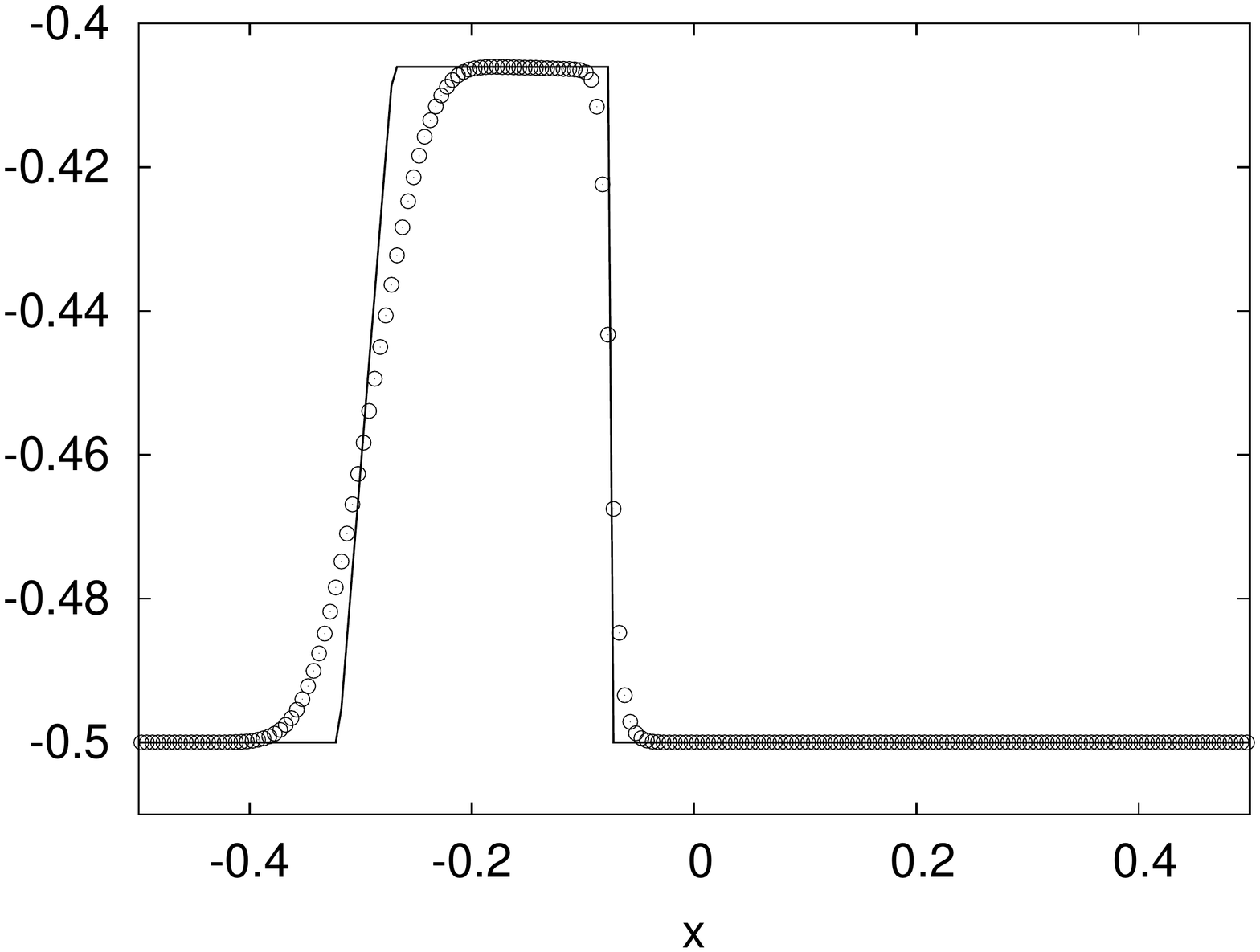  ,width=6cm}} \\ \vspace{-1cm}
\subfigure{\epsfig{figure=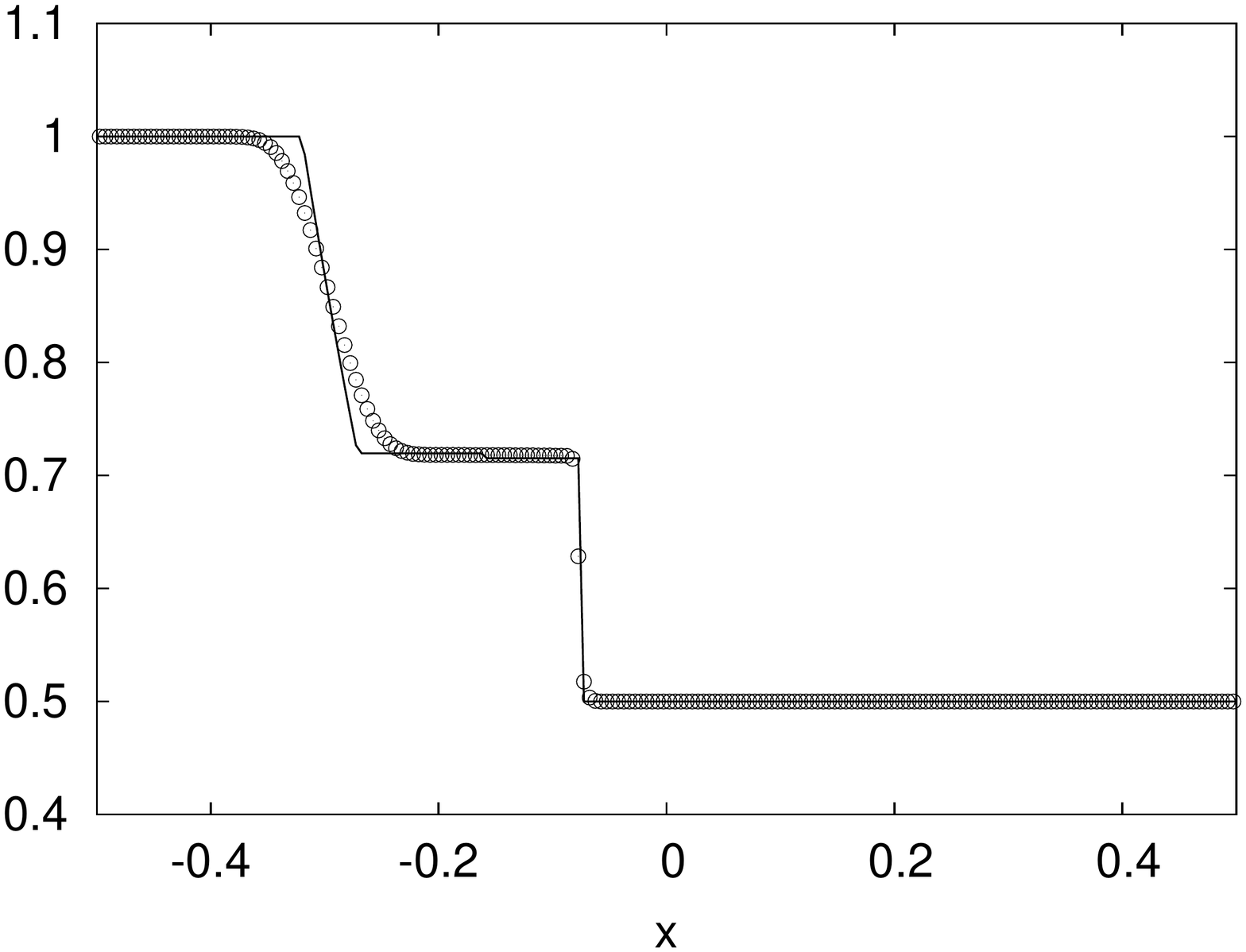,width=6cm}}
\subfigure{\epsfig{figure=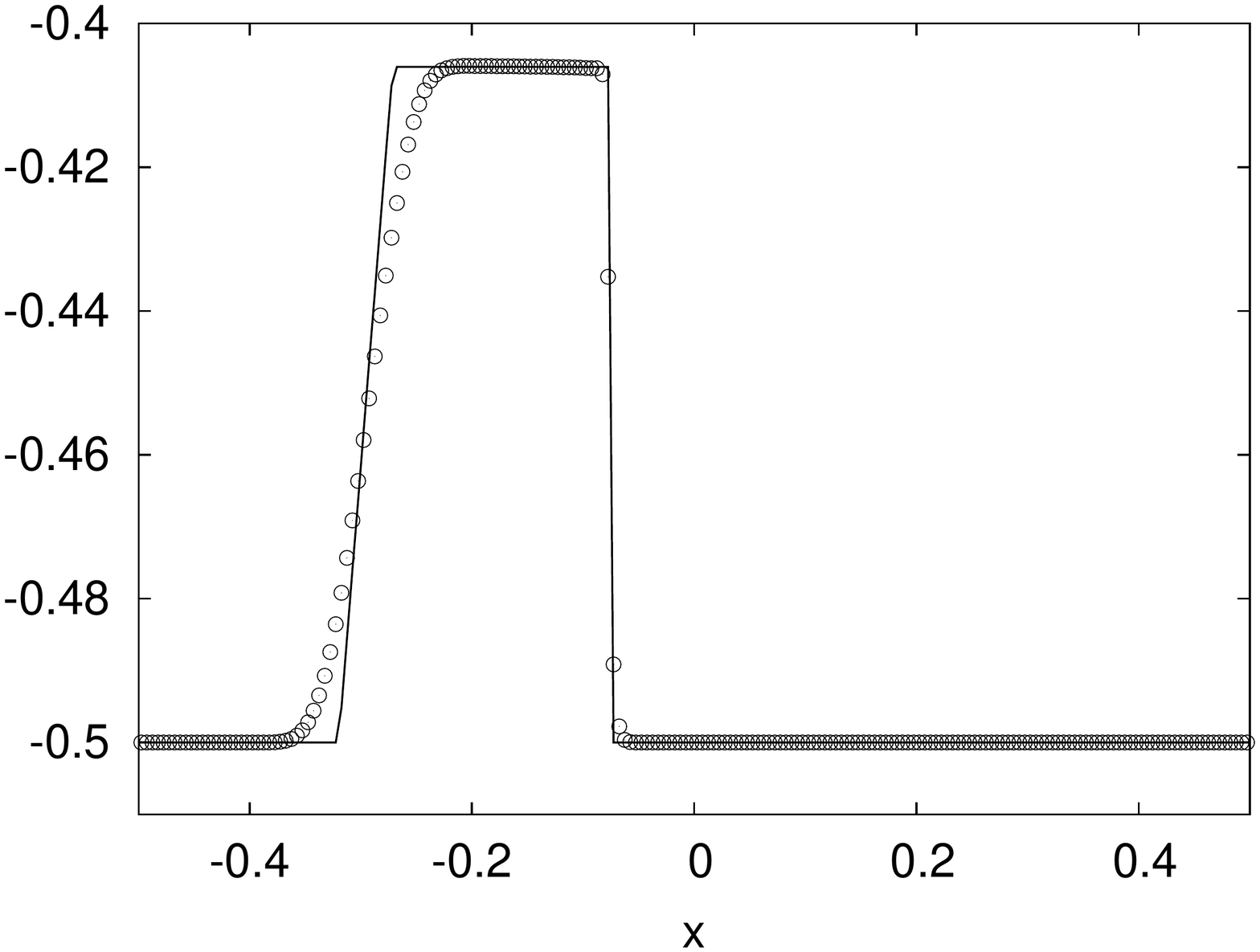  ,width=6cm}} \\ \vspace{-1cm}
\subfigure{\epsfig{figure=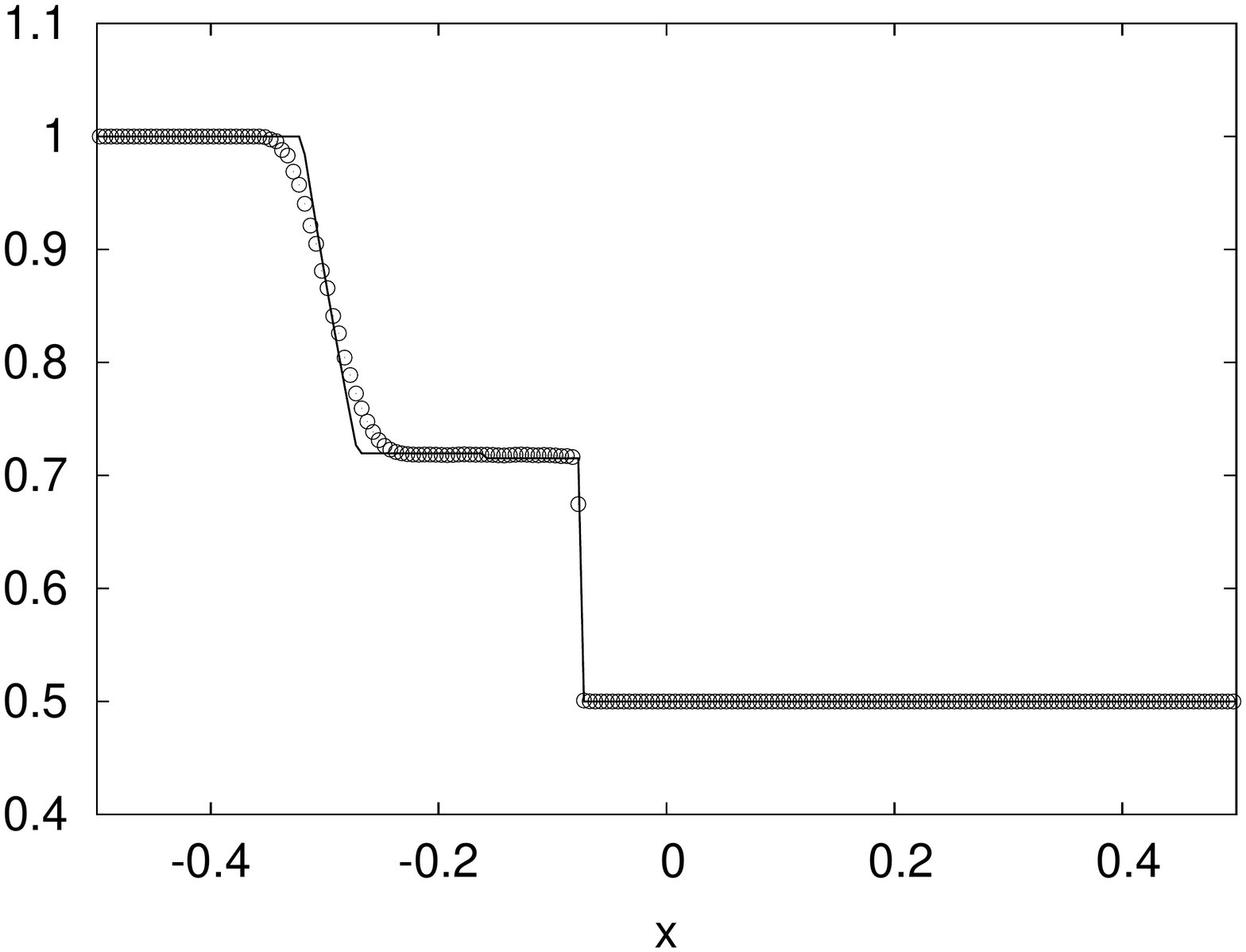,width=6cm}}
\subfigure{\epsfig{figure=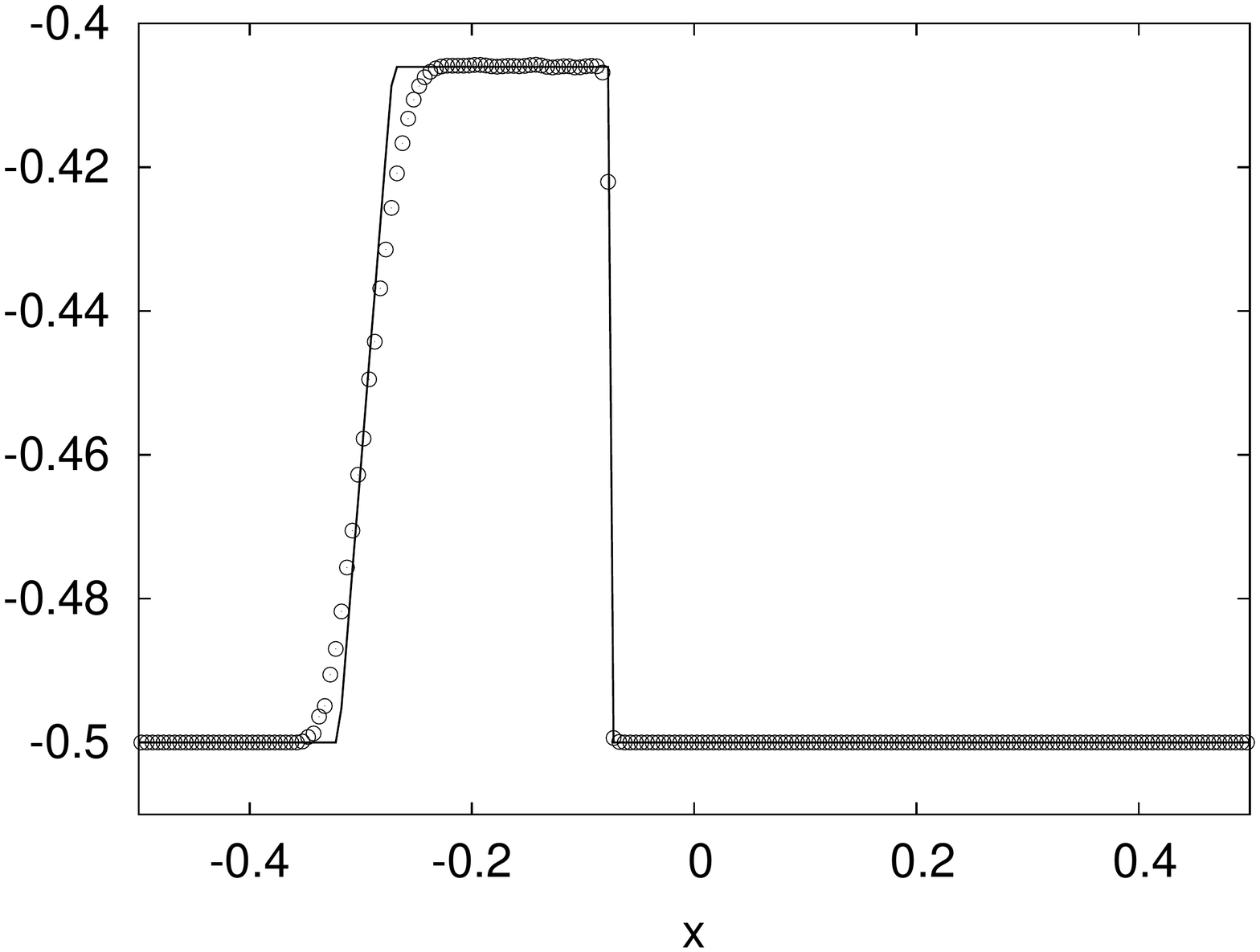  ,width=6cm}}
\caption{RP3: numerical solution for density (left) and velocity (right) at time $t=0.4$ for different polynomial degrees $p=1$ to $p=3$ from top to bottom, $h=\tfrac{1}{200}$.}
\label{fig:RP3}
\end{center}
\end{figure}

\begin{figure}
\begin{center}
\subfigure{\epsfig{figure=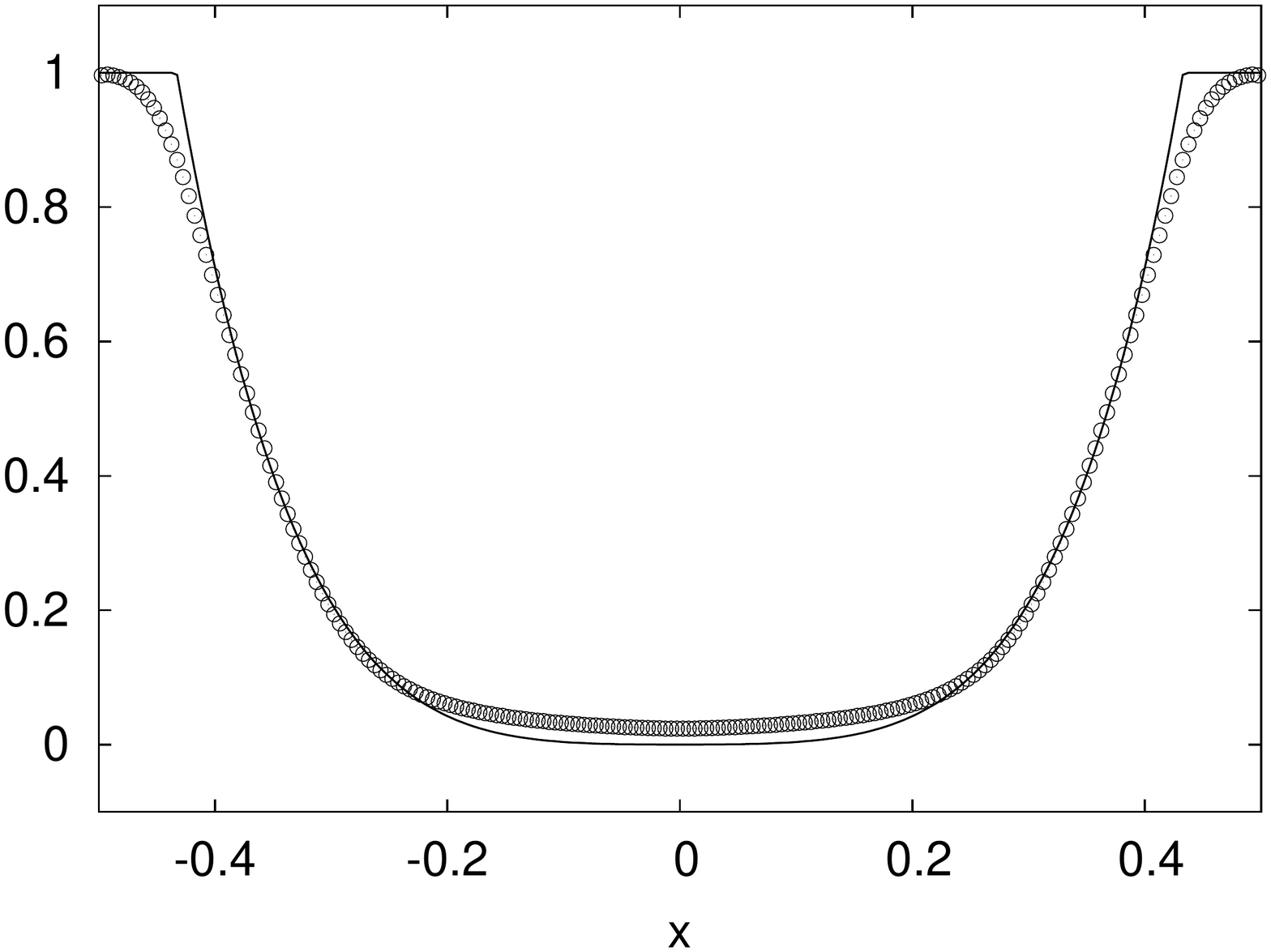,width=6cm}}
\subfigure{\epsfig{figure=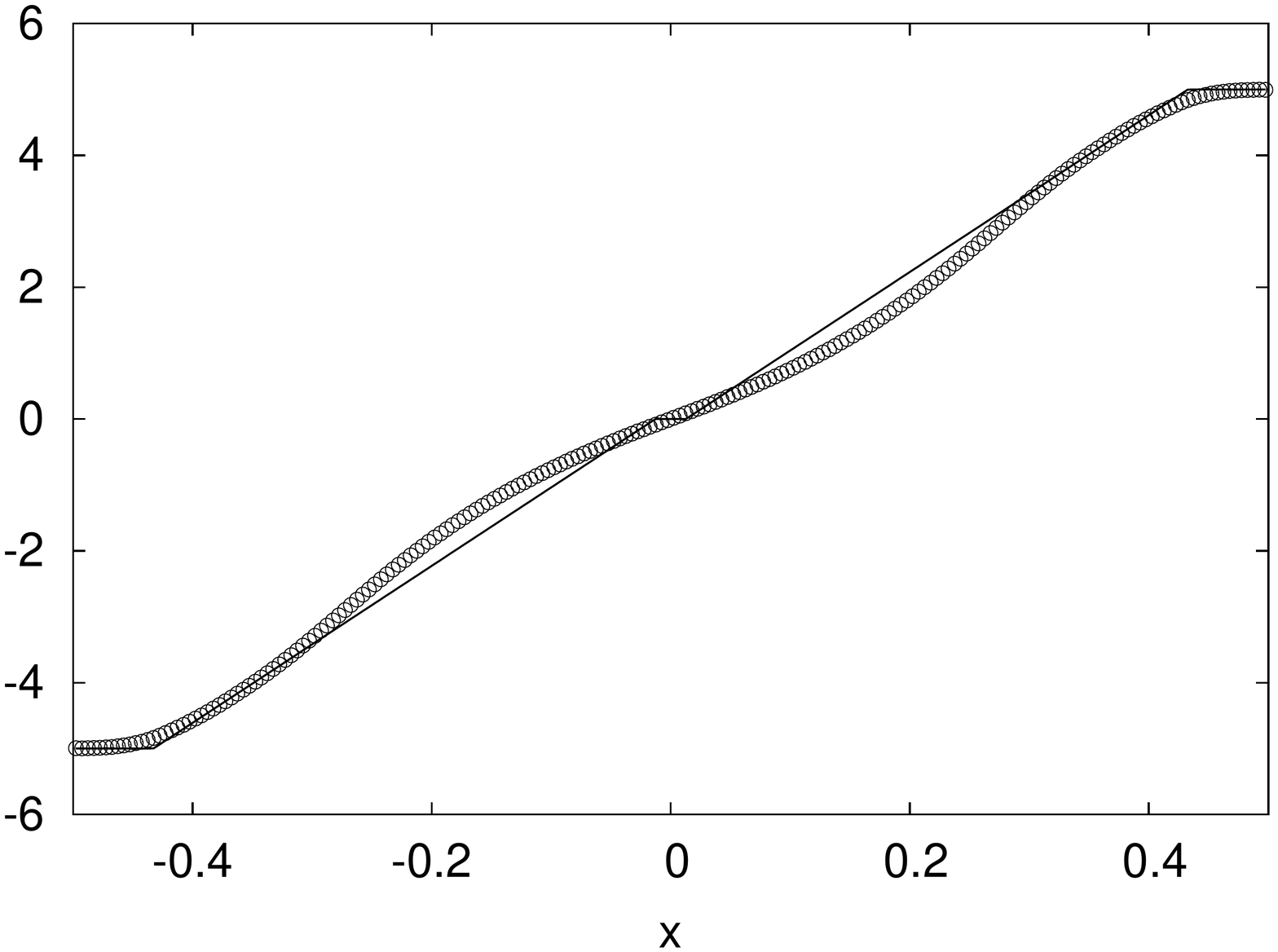  ,width=6cm}} \\ \vspace{-1cm}
\subfigure{\epsfig{figure=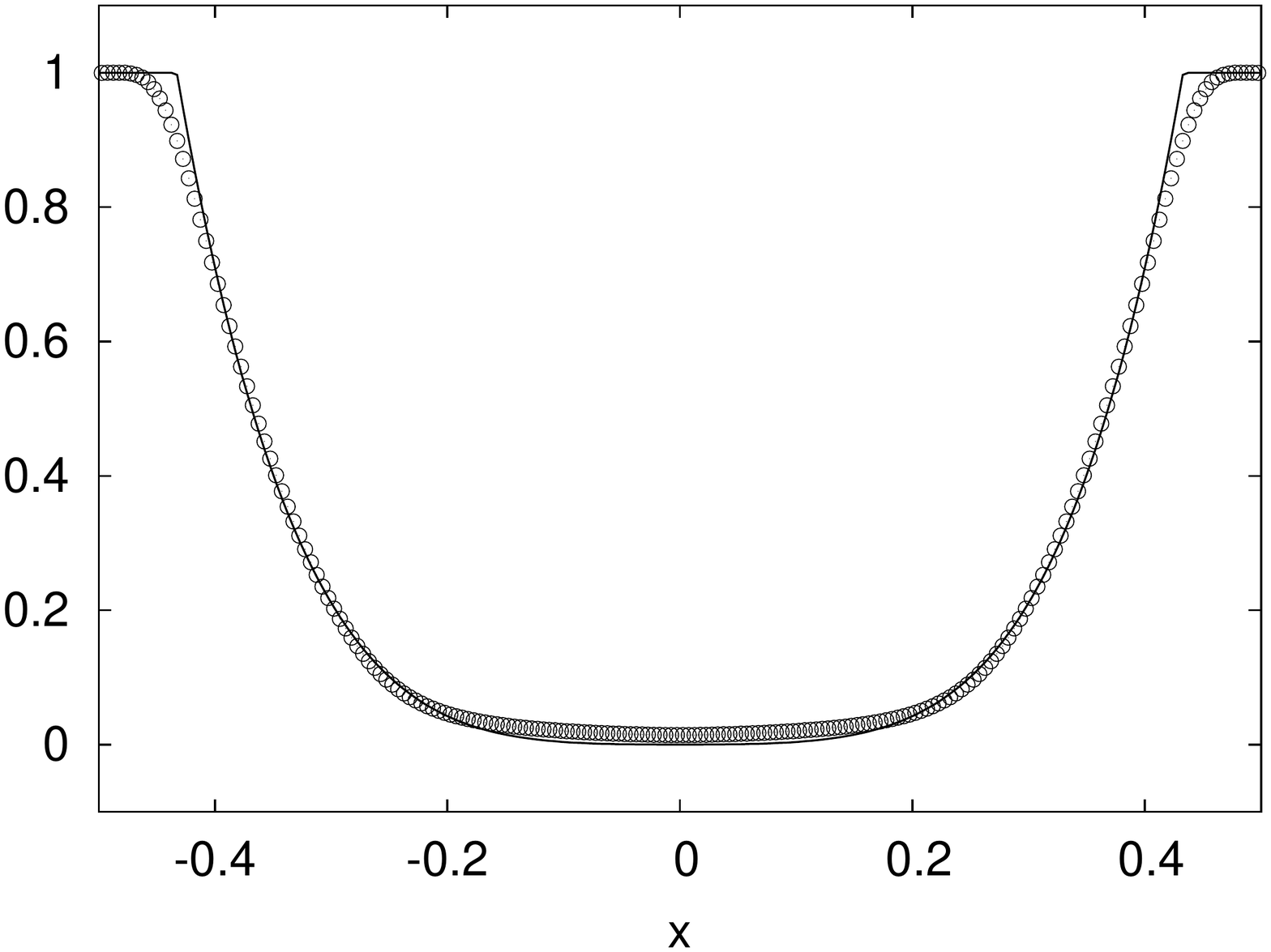,width=6cm}}
\subfigure{\epsfig{figure=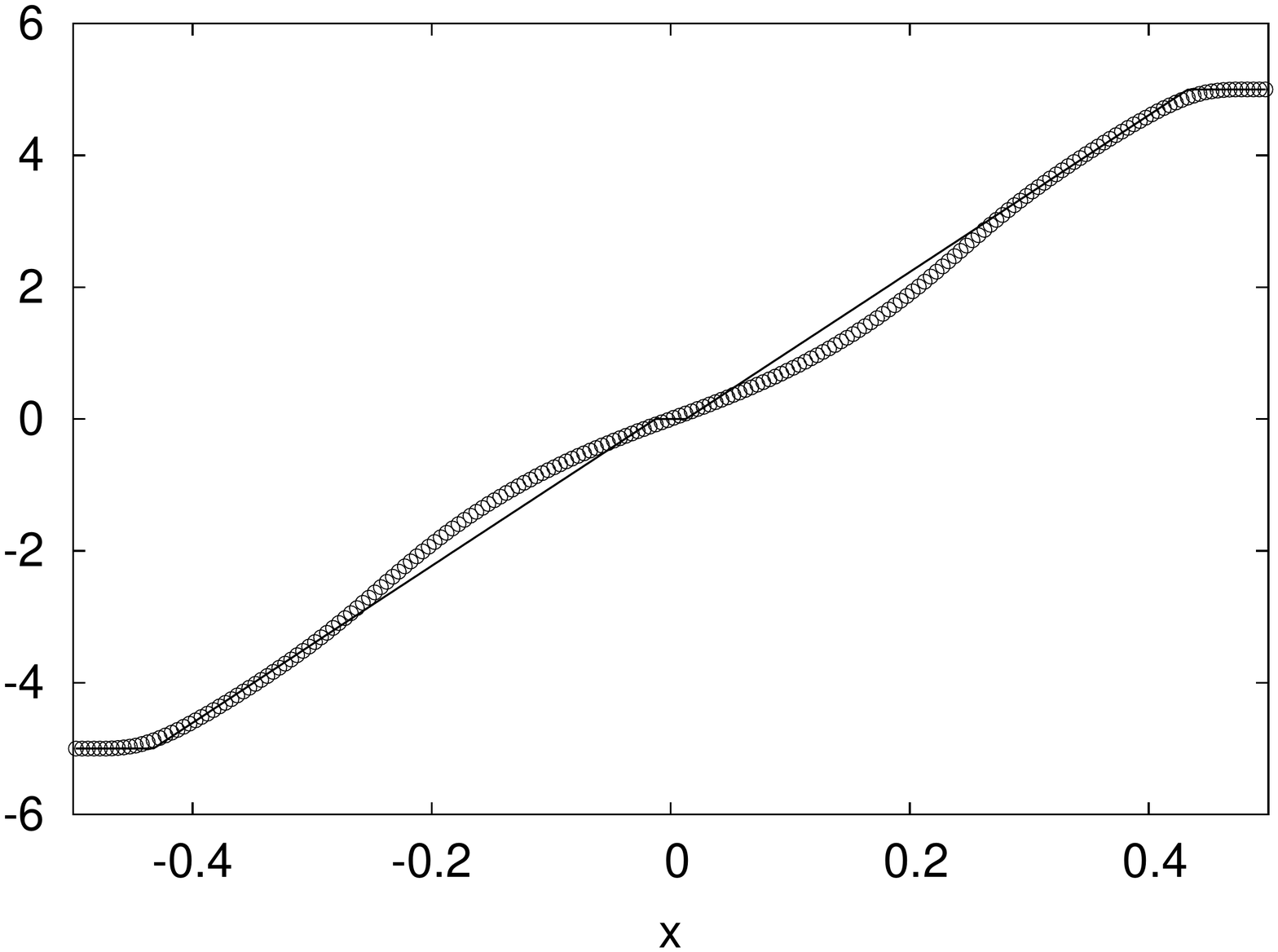  ,width=6cm}} \\ \vspace{-1cm}
\subfigure{\epsfig{figure=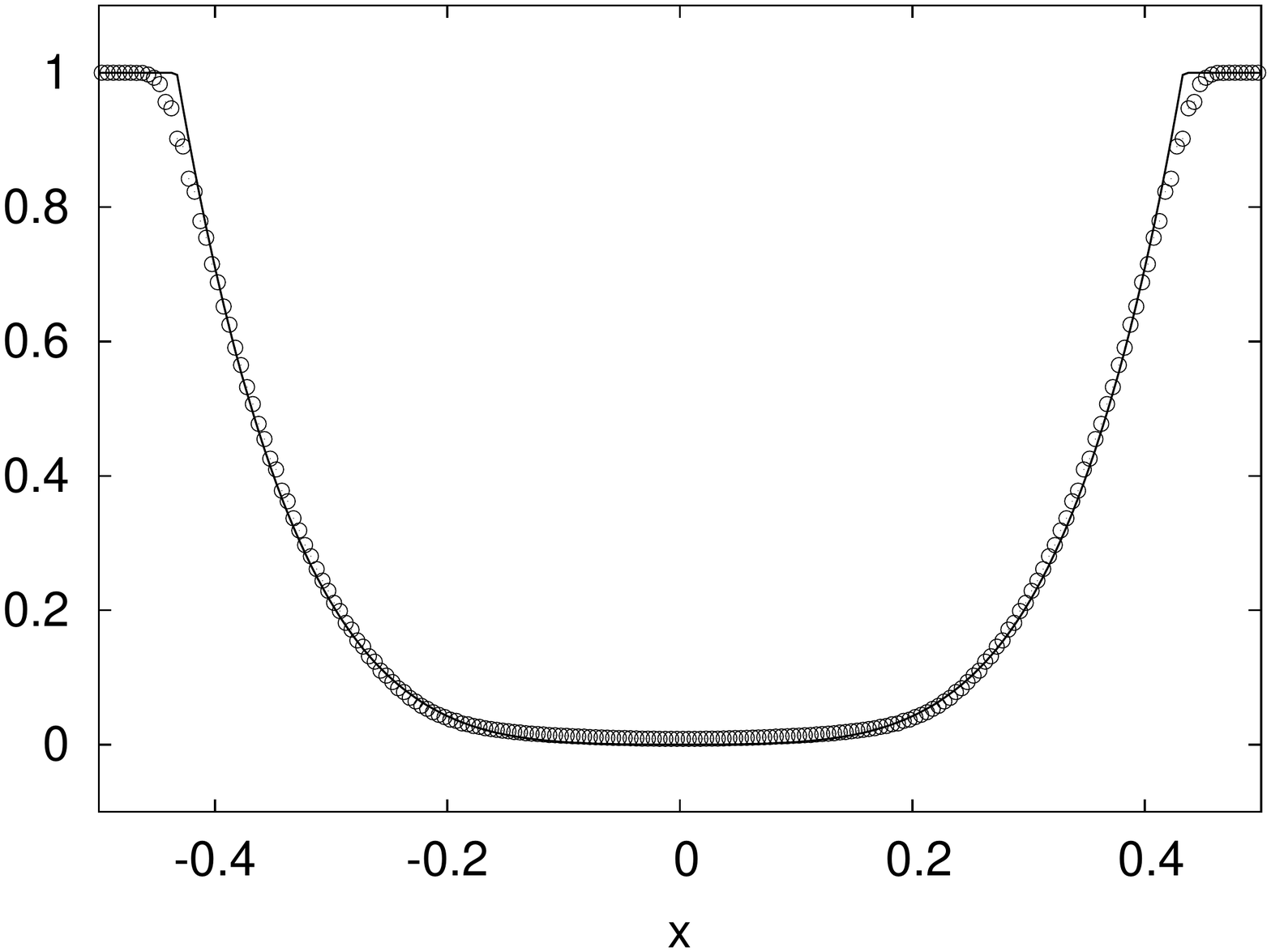,width=6cm}}
\subfigure{\epsfig{figure=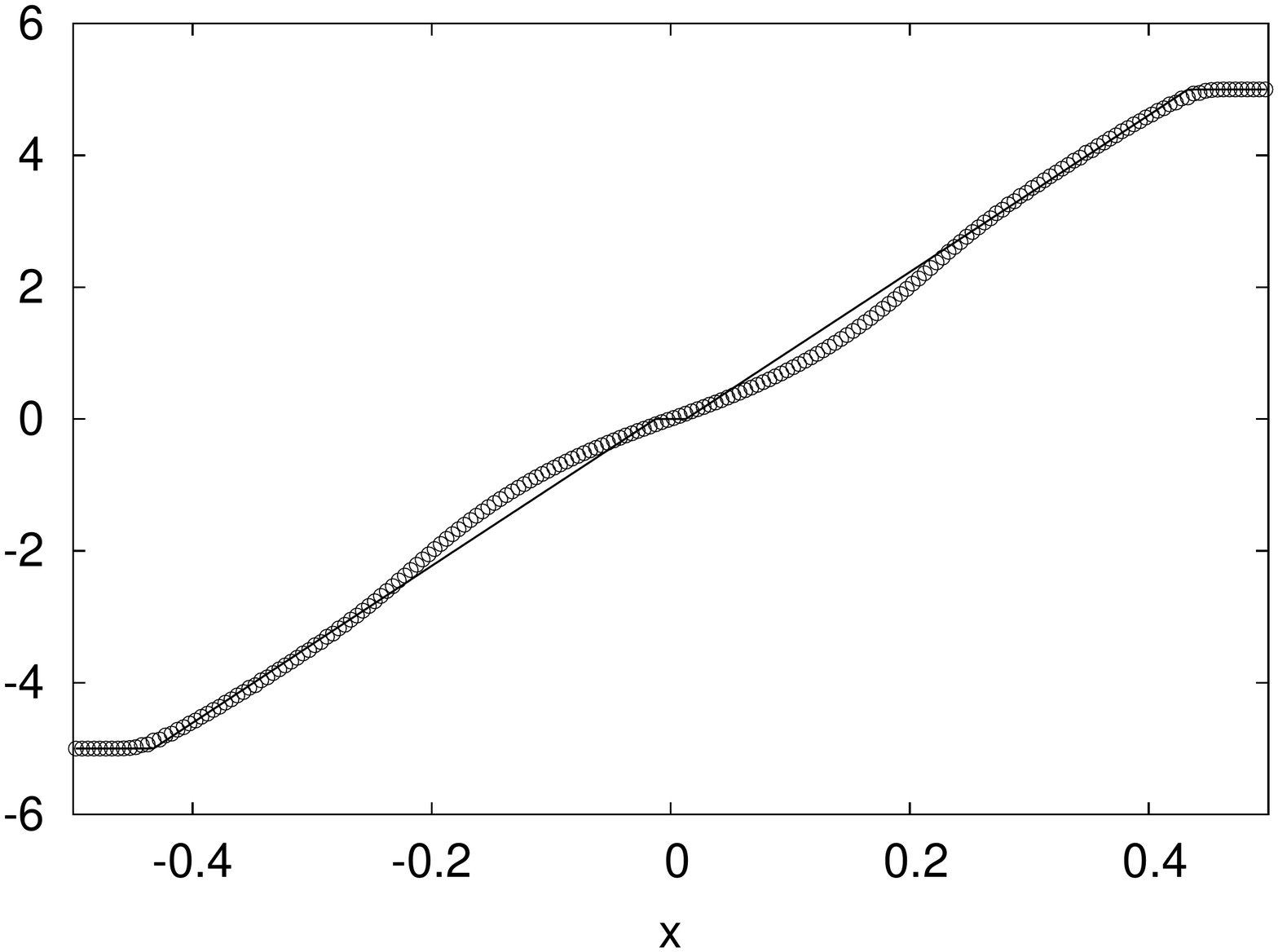  ,width=6cm}}
\caption{RP4: numerical solution for density (left) and velocity (right) at time $t=0.07$ for different polynomial degrees $p=1$ to $p=3$ from top to bottom, $h=\tfrac{1}{200}$.}
\label{fig:RP5}
\end{center}
\end{figure}

%
%
\section{Concluding remarks}\label{sec:conclusions}

The LPDG scheme introduced in this work is based on the Lagrange-projection like scheme from \cite{coquel_etal_10} derived in the context of a first-order finite volume formulation of the Euler equations. This method is here extended to high-order by using a DG method of arbitrary order for the space discretization and associated to a first-order implicit-explicit time discretization of acoustic and transport operators, respectively. Considering the isentropic Euler equations, \textit{a priori} conditions on the time step and on the numerical parameter imposing the subcharacteristic condition are derived in order to guaranty positivity and entropy inequality for the mean value of the numerical solution in each mesh element. {\it A posteriori} limiters similar to those introduced in \cite{zhang_shu_10a,zhang_shu_10b} are then used to extend these properties to nodal values within elements. Strong-stability preserving Runge-Kutta schemes are applied for the time integration in order to keep positivity at any time discretization order. 

Numerical experiments in one space dimension highlight high-order approximation of smooth solutions, while the method proves to be robust in the presence of discontinuities or vacuum. Future investigations will consider the full system of gas dynamics with a general equation of state and the extension to several space dimensions.

\begin{acknowledgements}
The author would like to thank Fr\'ed\'eric Coquel from \'Ecole Polytechnique and Christophe Chalons from Universit\'e de Versailles-Saint-Quentin-en-Yvelines for valuable discussions and their constructive comments.
\end{acknowledgements}



\end{document}